\documentclass{article}

\usepackage[utf8]{inputenc}
\usepackage[english]{babel}
\usepackage[left=3cm,right=3cm,top=3cm,bottom=3cm]{geometry}
\usepackage[colorlinks=true, allcolors=blue]{hyperref}
\usepackage{amsmath}
\usepackage{amsfonts}
\usepackage{amsthm}
\usepackage{subcaption}
\usepackage{graphicx}
\usepackage[ruled,vlined,french]{algorithm2e}
\usepackage{tikz}
\usepackage{float} 
\usepackage[justification=centering]{caption}
\usepackage{caption}
\usepackage{enumitem}

\usepackage{thmtools}
\usepackage{thm-restate}

\usepackage[sorting=none]{biblatex}
\usepackage{a4wide}
\usepackage{authblk}

\newtheorem{defn}{Definition}

\newtheorem{lemme}[defn]{Lemma}
\newtheorem*{lemme*}{Lemma}
\newtheorem{cor}[defn]{Corollary}

\newtheorem{theorem}[defn]{Theorem}

\newtheorem*{proposition*}{Proposition}
\newtheorem{conj}[defn]{Conjecture}
\newtheorem*{conj*}{Conjecture}
\newtheorem{question}[defn]{Question}

\newcommand{\ter}{ter} 

\title{Metric dimension on sparse graphs and its applications to zero forcing sets\footnote{An extended abstract of this work has been published in EuroComb 2021 \cite{eurocom2021}.} \footnote{The first three authors of this work are supported by ANR project GrR (ANR-18-CE40-0032). 

The Ignacio M. Pelayo work is supported by project PGC2018-095471-B-100.}}
\author[1]{Nicolas Bousquet}
\author[1]{Quentin Deschamps}
\author[1]{Aline Parreau}
\author[2]{Ignacio M. Pelayo}

\affil[1]{Universit\'e de Lyon, Universit\'e Lyon 1, LIRIS UMR CNRS 5205, F-69621, Lyon, France}
\affil[2]{Universitat Polit\`ecnica de Catalunya, Barcelona, Spain}

\begin{document}

\maketitle

\begin{abstract}
  The {\em metric dimension $\dim(G)$} of a graph $G$ is the minimum cardinality of a subset $S$ of vertices of $G$ such that each vertex of $G$ is uniquely determined by its distances to $S$. 
  It is well-known that the metric dimension of a graph can be drastically increased by the modification of a single edge.
  Our main result consists in proving that the increase of the metric dimension of an edge addition can be amortized in the sense that if the graph consists of a spanning tree $T$ plus $c$ edges, then the metric dimension of $G$ is at most the metric dimension of $T$ plus $6c$.

  We then use this result to prove a weakening of a conjecture of Eroh et al.
  The {\em zero forcing number} $Z(G)$ of $G$ is the minimum cardinality of a subset $S$ of {\em black} vertices (whereas the other vertices are colored white) of $G$ such that all the vertices will turned black after applying finitely many times the following rule: \emph{a white vertex is turned black if it is the only white neighbor of a black vertex.}
  
  Eroh et al. conjectured that, for any graph $G$, $\dim(G)\leq Z(G) + c(G)$, where $c(G)$ is the number of edges that have to be removed from $G$ to get a forest. They proved the conjecture is true for trees and unicyclic graphs.
  We prove a weaker version of the conjecture: $\dim(G)\leq Z(G)+6c(G)$ holds for any graph. We also prove that the conjecture is true for graphs with edge disjoint cycles, widely generalizing the unicyclic result of Eroh et al.
\end{abstract}

\section{Introduction}
\setcounter{page}{1}

A \emph{resolving set} of a graph $G$ is a subset of vertices of $G$ such that any vertex in the graph is identified by its distances to the vertices of the resolving set.
In the example of Figure~\ref{figexdim} the set $ \{u,v\}$ is a resolving set of the graph because all the vertices have a different \emph{distance vector} to $\{u,v\}$ so the knowledge of the distances to $u$ and $v$ uniquely identifies a vertex. This notion has been introduced in 1975 by Slater \cite{slater1975trees} for trees and by Harary and Melter \cite{harray1975} for graphs. Determining the minimum size of a resolving set, called \emph{metric dimension} and denoted by $\dim(G)$, is an NP-hard problem~\cite{garey1979} even restricted to planar graphs~\cite{Diaz2012}. Applications of metric dimension go from piloting a sonar~\cite{harray1975} to the navigation of a robot in an Euclidean space~\cite{KHULLER1996}.

\begin{figure}[ht]
    \centering
    \includegraphics[scale=1.7]{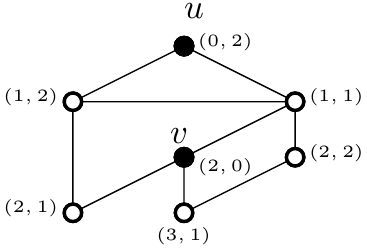}
    \caption{The black vertices form a resolving set. For each vertex $x$ the vector next to $x$, $(d(u,x),d(v,x))$ is unique.}
    \label{figexdim}
\end{figure}

One of the main issues to compute the metric dimension of a graph comes from the fact that it is unstable when the graph is modified. When a vertex is added, the metric dimension can be drastically modified. Indeed, while a path admits a constant size resolving set, a path plus a universal vertex only admits resolving sets of linear size (in the number of vertices). However, in this example even if only one vertex was added, a linear number of edges were also added which had permitted to put all the vertices at distance $2$. One can then wonder if the situation is better when only one edge is modified in the graph. Unfortunately again, the metric dimension of a graph can be drastically modified by the modification of a single edge. In Figure \ref{fig_delete_edge} (first proposed in~\cite{eroh2015effect}), the metric dimension of the left graph is $2k$ where $k$ is the number of layers in the graph while the right graph has metric dimension $k+1$. So the addition of one edge can modify the metric dimension by $\Omega(n)$.

\begin{figure}[ht]
    \centering
    \includegraphics[scale=0.7]{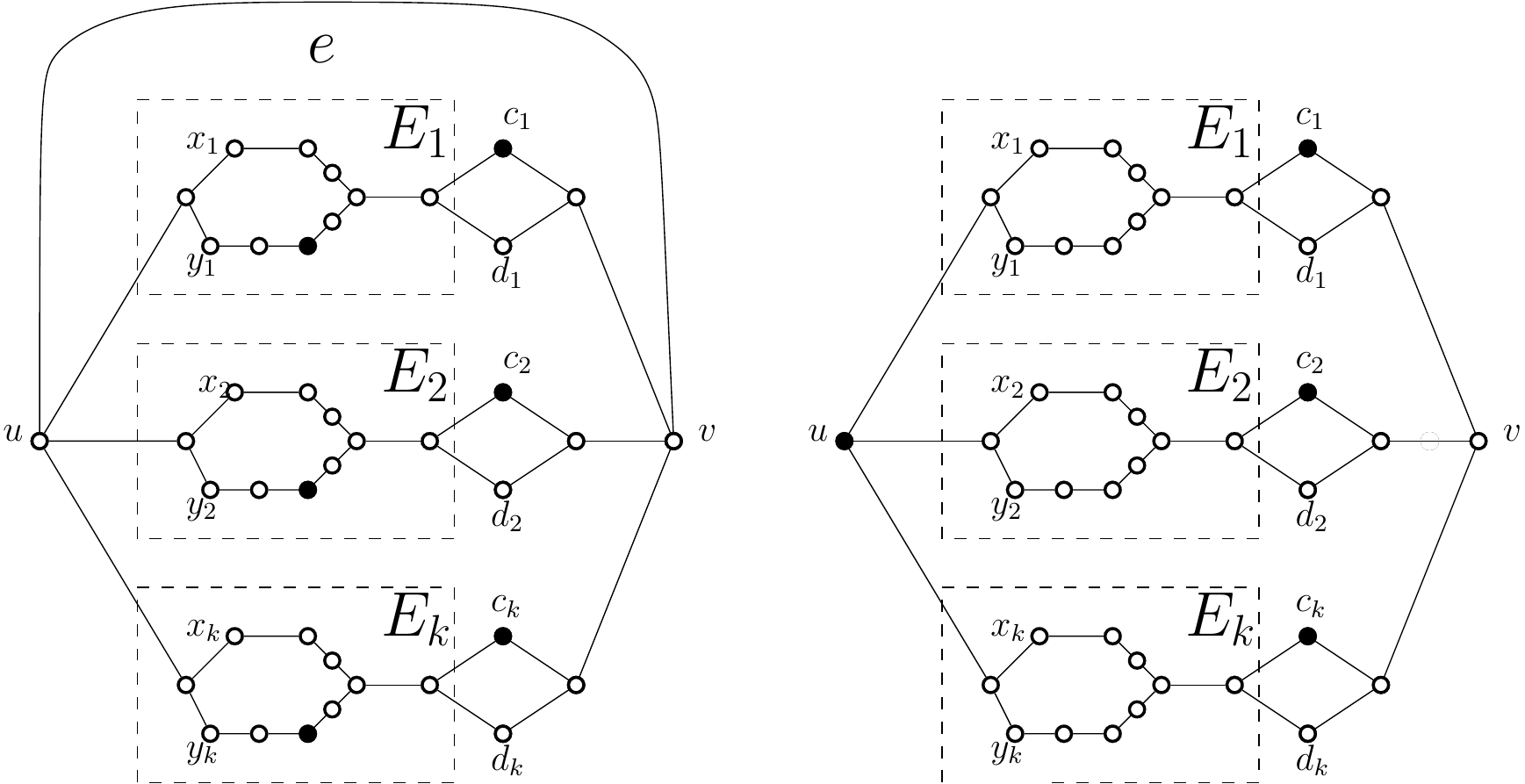}
    \caption{\raggedright{On the left graph, all the pairs $(c_i,d_i)$ are twins (both vertices have the same neighborhood) so for all $i \leq k$ any resolving set should contain $c_i$ or $d_i$. Then, for $i \leq k$, the pair $(x_i,y_i)$ can only be resolved by a vertex of $E_i$ so  any resolving set should contain a vertex of $E_i$ for all $i \leq k$. So any resolving set contains at least $2k$ vertices. The set of black vertices is a resolving set so the dimension is $2k$. On the right graph the pairs $(c_i,d_i)$ are twins so for all $i \leq k$ any resolving set should contain $c_i$ or $d_i$. One can easily check that a set containing only vertices on $\{c_i;d_i, 1 \leq i \leq k\}$ is not resolving so the dimension is higher than $k$. Black vertices form a resolving set so the metric dimension is $k+1$. }}
    
    \label{fig_delete_edge}
\end{figure}

\paragraph{Metric dimension and cycle rank.}
 Eroh et al.~\cite{eroh2017comparison} proved that, if $G$ is a unicyclic graph, then the metric dimension of $G$ is at most the metric dimension of any spanning tree of $G$ plus one. Moreover, all the existing examples where the metric dimension is drastically modified with an edge modification already contain many cycles. One can then wonder if the metric dimension of a graph which does not contain many cycles is close to the metric dimension of any of its spanning trees. The goal of this paper is to answer this question positively. 
  
One can easily remark that in the star $K_{1,n}$, any resolving set contains at least $n-1$ vertices. Indeed otherwise, two degree $1$ vertices are not in the resolving set and then these two vertices cannot be distinguished. This argument can be generalized to any graph as follows: if $r$ pending paths are attached to a vertex $v$, any resolving set contains a vertex in at least $r-1$ of them. Let $L(G)$ be the sum over all the vertices $v$ on which there are attached pending paths, of the number of paths attached to $v$ minus one per vertex $v$.
 Chartrand et al.~\cite{CHARTRAND2000} remarked that, for every connected graph $G$, $\dim(G) \geq L(G)$ with equality for trees. However, this bound has no reason to be closed to the optimal value (a graph with no degree $1$ vertex can have an arbitrarily large metric dimension). A wheel (induced cycle plus a universal vertex), for instance, has no degree $1$ vertex and its metric dimension is linear in $n$ (see Figure \ref{figwheel}).

The \emph{cycle rank} of a connected graph, denoted by $c(G)$, is the number of edges that has to be removed from $G$ to obtain a spanning tree.  We prove that the following holds:
 
 \begin{restatable}{theorem}{mainbisstate}
\label{thm:mainbis}
For any graph $G$, $L(G) \leq \dim(G) \leq L(G) + 6c(G) $.
\end{restatable}
 
 Since the value of $L(G)$ cannot decrease by removing edges in $G$ that are not bridges, it implies the following:

  \begin{restatable}{cor}{corST}
\label{cor:ST}
 Let $G$ be a graph and $T$ be any spanning tree of $G$. Then,
 \[ \dim(G) \le \dim(T)+6c(G) \]
 \end{restatable}

Informally, Corollary~\ref{cor:ST} ensures that, even if the metric dimension can be widely modified when we add a single edge, the ``amortized cost" of an edge addition is at most $6$ with respect to any spanning tree of $G$.
As far as we know it is the first bound of the metric dimension in terms of the natural lower bound $L(G)$ or of the metric dimension of a spanning tree of $G$. 
 
 Before explaining briefly the outline of the proof, let us discuss a bit the tightness of these results.
Let $k \in \mathbb{N}$. Consider the graph $G_k$ which is a collection of $k$ $C_4$ glued on the central vertex of a path of length $3$ like in Figure \ref{fig2}. The metric dimension of $G_k$ is equal to $2k+1$ (we need to select exactly two vertices per $C_4$ distinct from the center and one extremity of the path), $L(G_k)=1$, and $c=k$. Since $L(G_k)=1$ for every $k$, there exist graphs $G$ such that $\dim(G) = L(G)+2c(G)$. We ask the following question: 

 \begin{question}\label{ques:L+c}
Is it true that for any graph $G$, $\dim(G) \leq L(G) + 2c(G) $ ?
 \end{question}
 
Note that Sedlar and Skrekovski independently ask the same question in \cite{sedlar2021vertex}. The same authors also prove in~\cite{SEDLAR2021} that Question~\ref{ques:L+c} is true for cacti (and tight since $G_k$ is a cactus). 

Let us now discuss the tightness of Corollary~\ref{cor:ST}. If, in the graph $G_k$, we remove one edge of each $C_4$ incident to the central vertex, the resulting spanning tree has metric dimension of order $k+1$ as long as $k \ge 2$. So there exist graphs $G$ for which there exists a spanning tree $T$ satisfying $\dim(G) = \dim(T)+c$. 
 We actually ask the following question:
 
 \begin{question}\label{ques:dimGT}
 Is it true that for any graph $G$ and for every spanning tree $T$ of $G$, we have $\dim(G) \le \dim(T)+c$? 
 \end{question}
 
One can then wonder what happens if we select the best possible spanning tree to start with, i.e., the spanning tree that maximizes the metric dimension. In $G_k$, one can note that if we break the edges of the $C_4$s that are not incident to the center and denote by $T_k$ the resulting tree, then $\dim(G_k)=\dim(T_k)+1$. Surprisingly, we did not find any graph $G$ where the metric dimension is a function of $c$ larger than \textit{any} spanning tree of $G$. We left the existence of such a graph as an open problem.
  \medskip
  
  Let us now briefly discuss the main ingredients of the proof of Theorem~\ref{thm:mainbis}. First, it consists in finding a small feedback vertex set $X$ of the graph such that every connected component is attached to at most two vertices of $X$. We then prove that, if we add to a resolving set of every connected component of $G \setminus X$ few vertices (in terms of $c$) we can ``detect" shortcuts passing through the rest of the graph and then obtain a resolving set of the whole graph $G$. The proof of Theorem~\ref{thm:mainbis} is given in Section~\ref{section3}.

   \medskip

The second part of the paper consists in applying this result in order to prove a weak version of a conjecture linking metric dimension and zero forcing sets in graphs.

\paragraph{Zero forcing sets.}
A {\em zero forcing set} is a subset of vertices colored in black which colors the whole vertex set in black when we apply the following rule: \emph{A vertex is colored black if it is the unique non-black neighbor of a black vertex}. See Figure~\ref{figexz} where the initial set contains three black vertices. The \emph{zero forcing number} of a graph is the minimal size of a zero forcing set, denoted by $Z(G)$. The zero forcing number has been introduced in 2008 to bound the rank of some families of adjacency matrices \cite{AIM2008}. Deciding if the zero forcing number of a graph is at most $k$ is NP-complete~\cite{TREFOIS2015}.

\begin{figure}[ht]
    \centering
    \includegraphics[scale=1]{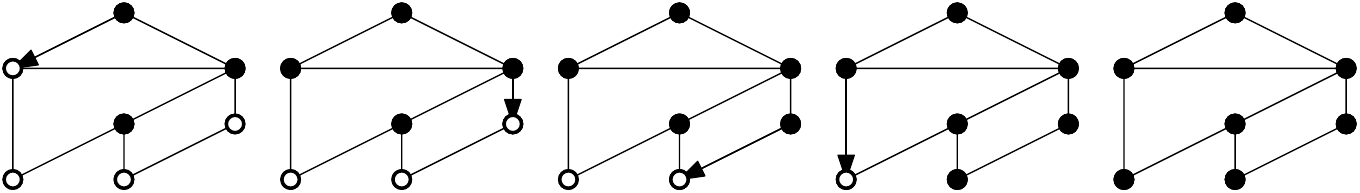}
    \caption{Iterations of the change color rule. On the graph on the left, the three black vertices form a zero forcing set.}
    \label{figexz}
\end{figure}

In general, the gap between metric dimension and zero forcing number can be arbitrarily large. But for some restricted sparse graph classes like paths or cycles, both the optimal parameters and optimal sets are the same. Eroh, Kang and Yi then started a systematic comparison between them~\cite{eroh2017comparison}. They proved that $\dim(G)\leq Z(G)$ when $G$ is a tree and that $\dim(G)\leq Z(G)+1$ when $G$ has one cycle (in other words $G$ is a tree plus an edge). On the other hand, $\dim(G)$ can be arbitrarily larger than the zero forcing number when the number of cycles increases. They conjectured the following:

\begin{conj}[Cycle-rank conjecture \cite{eroh2017comparison}]\label{conj}
For a connected graph, $\dim(G) \leq Z(G) +c(G)$.
\end{conj}

Conjecture~\ref{conj} is tight for an infinite family of graphs: The graph $G_k$ contains a path of $3$ vertices and $k$ cycles of size $4$ with the central vertex of the path in common. Figure~\ref{fig2} shows the graph $G_3$ with $c(G_3)=3$.

\begin{figure}[ht]
    \centering
    \includegraphics[scale=1.2]{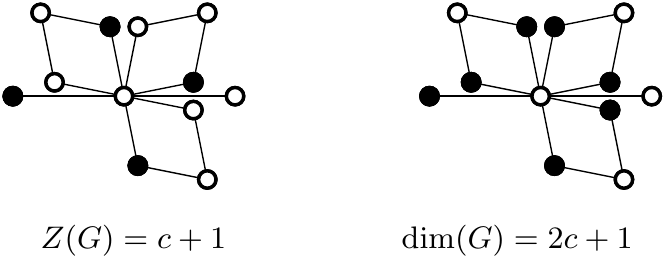}
    \caption{Tightness of Conjecture~\ref{conj}.}
    \label{fig2}
\end{figure}

Eroh et al. proved in~\cite{eroh2017comparison} that $\dim(G)\leq Z(G) + 2c(G)$ if $G$ contains no even induced cycles.
Our main contribution to this question is to prove a weaker version of Conjecture~\ref{conj} in Section~\ref{sec:consZF}, whose proof is mainly based on an application of Theorem~\ref{thm:mainbis}.

  \begin{restatable}{theorem}{thmZFdim}
\label{thm:main}
For every graph $G$, we have 
\[\dim(G)\leq Z(G)+6 c(G). \]
\end{restatable}

As far as we know, it is the first upper bound of $\dim(G)$ of the form $Z(G)+f(c(G))$. 

Note that the dependency on $c(G)$ cannot be removed, i.e.,  $\dim(G)$ cannot be upper bounded by a function of $Z(G)$ only. For the wheel of $n$ vertices (a cycle plus a universal vertex, see Figure \ref{figwheel}), the zero forcing number is $3$ for any $n \geq 4$ but the metric dimension is a linear function in $n$. 
\begin{figure}
    \centering
    \includegraphics[scale=1]{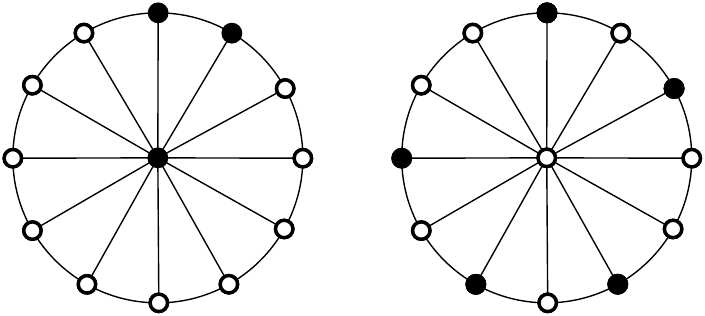}
    \caption{The zero forcing number of a wheel is $3$ (while $n \geq 4$) but its metric dimension is linear in $n$.}
    \label{figwheel}
\end{figure}

We also prove Conjecture~\ref{conj} in several particular cases. 
We first focus on unicyclic graphs. We give an alternative proof of Conjecture~\ref{conj} for unicyclic graphs with a much shorter and simpler proof than the one of~\cite{eroh2017comparison}.
We then extend our results to prove Conjecture~\ref{conj} for cactus graphs\footnote{This result is proved independently in \cite{sedlar2021vertex} with a different method.} (graphs with edge-disjoint cycles).
It generalizes the result on unicyclic graphs and is based on a very simple induction whose base case is the case of unicyclic graphs. Since cactus graphs contain the class of graphs with no even cycles, it improves the result of~\cite{eroh2017comparison} on even-cycle-free graphs.

We finally show that $\dim(G) \leq Z(G)$ when the unique cycle of $G$ has odd length. This result is tight and cannot be extended to unicyclic graphs with an even cycle as shown in Figure \ref{figzdim}. All the results related to zero forcing sets are proved in Section~\ref{section4}.

 \begin{figure}[ht]
    \centering
    \includegraphics[scale=1.2]{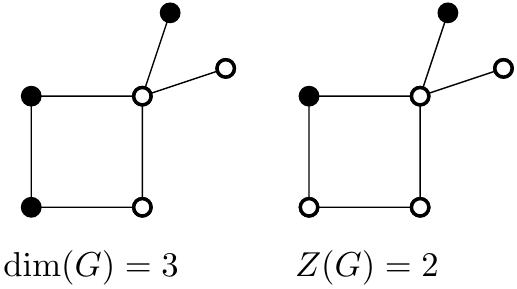}
    \caption{Black vertices form respectively a minimal resolving set and a minimal zero forcing set.}
    \label{figzdim}
\end{figure}

\section{Preliminaries} \label{section2}
\subsection{Definitions and notations}
Unless otherwise stated, all the graphs considered in this paper are undirected, simple, finite and  connected. For standard terminology and notations on graphs, we refer the reader to \cite{bookgraph}.

Let $G=(V,E)$ be a graph. The \emph{distance} between two vertices $u,v \in V$, denoted by $d_G(u,v)$ (or simply $d(u,v)$ when $G$ is clear from context), is the length of a shortest path from $u$ to $v$ in $G$. When no such path exists, we state $d_G(u,v)=+\infty$.
For $v \in V$, let $N(v)$ be the (open) neighborhood of $v$ defined as $N(v) = \{u \in V, \; uv \in E\}$. We say that two vertices $v$ and $w$ are \emph{twins} if $N(v)\setminus \{w\} =N(w)\setminus \{v\}$. For $X \subseteq V$, let $G[X]$ be the subgraph of $G$ induced by $X$. In other words, $G[X]$ is the graph with vertex set $X$ where $xy$ is an edge if and only if it is an edge of $G$. We denote by $G \setminus X$ the subgraph of $G$ induced by $V \setminus X$. The \emph{border of $X$}, denoted by $\partial X$, is $\{ u \in G \setminus X | \; \exists v\in X, uv \in E \}$.
 
A vertex $w \in V$ \emph{resolves} a pair of vertices $(u,v)$ if $d(w,u) \neq d(w,v)$. Let $S \subseteq V$. The set $S$ \emph{resolves} the pair $(u,v)$ if at least one vertex in $S$ resolves the pair $(u,v)$ and $S$ resolves a set $W \subseteq V$ if $S$ resolves all the pairs of $W$. A set $S$ is a \emph{resolving set} of $G$ if $S$ resolves $V$. \emph{The metric dimension} $\dim(G)$ of $G$ is the minimum cardinality of a resolving set in $G$. A resolving set of minimum size is called a \emph{metric basis}.

Let $Z\subseteq V$ be a set of vertices. The vertices in $Z$ are colored in black whereas the other vertices are white. The \emph{color change rule} converts a white vertex $u$ into a black vertex if $u$ is the only white neighbor of a black vertex. The set $Z$ is a \emph{zero forcing set} of $G$ if all the vertices of $G$ can be turned black after finitely many applications of the color change rule. For $u$ and $v$ two vertices in $V$ and a sequence of applications of the color change rule, we say that $u$ \emph{forces} $v$ if at some step $u$ is turned black with the color change rule because of $v$. We say that the edge $uv$ is used to force $u$.  \emph{The zero forcing number} $Z(G)$ of $G$ is the minimal cardinality of a zero forcing set in $G$.

 The \emph{cycle rank} of $G$, denoted by $c(G)$ (or $c$ if the context is clear enough), is the minimum number of edges that should be deleted from $G$ to get a forest. Note that we have $c(G)=|E|-|V|+cc(G)$ where $cc(G)$ is the number of connected components of $G$. A graph $G$ is \emph{unicyclic} if $G$ is connected with $c(G)=1$.
 A {\em feedback vertex set} of $G$ is a subset of vertices such that $G \setminus X$ is a forest.  We denote by $\tau(G)$ (or $\tau$ if the context is clear enough) the minimum size of a feedback vertex set of $G$. Note that if $X$ has minimum size, then $\tau(G)\leq c(G)$.

\subsection{Resolving sets and zero forcing sets on trees}

Chartrand et al. \cite{CHARTRAND2000} introduced the following terminology to study resolving sets in trees.
We extend this terminology to general graphs (see Figure~\ref{figtreename} for an illustration).\\
A vertex of degree $1$ is called a \emph{terminal vertex}.\\
A vertex of degree at least 3 is a \emph{major vertex}. A terminal vertex $u$ is called \emph{a terminal vertex of a major vertex $v$} if $d(u,v)<d(u,w)$ for every other major vertex $w$. In other words $u$ and $v$ are linked by a path of degree $2$ vertices. The \emph{terminal degree} of a major vertex $v$ is the number of terminal vertices of $v$, denoted by $\ter(v)$. A major vertex is \emph{exterior} if its terminal degree is positive, and \emph{interior} otherwise. \\
A degree-$2$ vertex is an \emph{exterior degree-2 vertex} if it lies on a path between a terminal vertex and its major vertex. It is an \emph{interior degree-2 vertex} otherwise.

\begin{figure}[ht]
    \centering
    \includegraphics[scale=0.9]{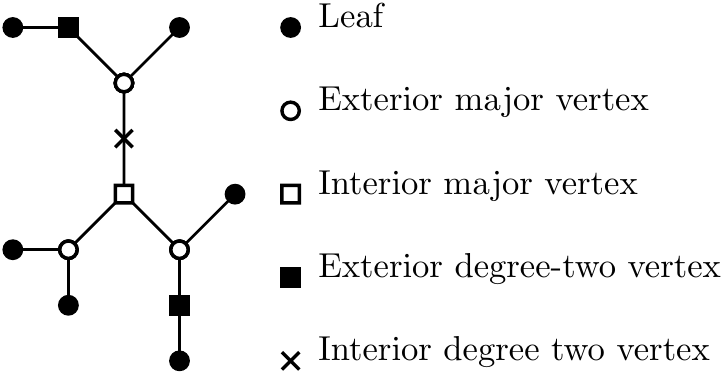}
    \caption{Vertices denomination in a graph.}
    \label{figtreename}
\end{figure}

Let $\sigma(G)$ be the sum of the terminal degrees over all the major vertices in $G$ and $ex(G)$ be the number of exterior major vertices in $G$. Let $L(G)=\sigma(G)-ex(G)$. If $G$ is a path $P_n$ for some $n \geq 1$, let $L(G)=1$.
We can bound $\dim(G)$ and $Z(G)$ with this parameter:

\begin{lemme}\cite{eroh2017comparison} \label{lowerbound}
 For any connected graph $G$, $\dim(G) \geq L(G)$ and $Z(G) \geq L(G)$.
\end{lemme}

\begin{lemme}\cite{eroh2015effect}  \label{dimtree}
Let $T$ be any tree, then, $dim(T)=L(T)$. Moreover, if $T$ is not a path, any set containing all but exactly one terminal vertices of every major vertex is a resolving set of~$T$.
\end{lemme}

There is no similar result on the zero forcing number of a tree. 
The gap between the zero forcing number and the metric dimension can be arbitrarily large on trees. 

Lemmas~\ref{lowerbound} and~\ref{dimtree} imply that trees satisfy Conjecture~\ref{conj}. Moreover, the equality case can be characterized:

\begin{lemme}\cite{eroh2017comparison}\label{treecomp}
For every tree $T$, $\dim(T) \leq Z(T)$. The equality holds if and only if $T$ has no interior degree-2 vertices and each major vertex has terminal degree at least two.
\end{lemme}

\subsection{Elementary results on metric dimension}
This section is devoted to some elementary results about metric dimension.

\begin{lemme}
\label{technical}
Let $G$ be a graph and $u$, $v$ be two vertices of $G$. Let $s$ and $t$ be two different vertices on a shortest path between $u$ and $v$. Then, $d(u,s) \neq d(v,s)$ or $d(u,t) \neq d(v,t)$.
\end{lemme}
\begin{proof}
Let $P$ be a shortest path from $u$ to $v$ containing $s$ and $t$. Up to symmetry, we can assume that $u,s,t,v$ appear in that order in $P$. Since $P$ is a shortest path, $d(s,v)=d(s,t)+d(t,v)>d(t,v)$ and $d(t,u)=d(s,u)+d(s,t)>d(s,u)$. Assume that $d(u,s)=d(v,s)$.  
Then, $d(u,t)=d(u,s)+d(t,s)$ and $d(v,s)=d(v,t)+d(t,s)$. So, $d(u,t)=d(v,t)+2d(t,s) \neq d(v,t)$ as $t \neq s$.
\end{proof}

\begin{lemme}
\label{cycle1}
Let $G$ be a unicyclic graph with a cycle $C$ of odd length. Then, every pair of vertices of $C$ resolves $C$.
\end{lemme}

\begin{proof}
Let $u$ and $v$ be two vertices of $C$. There are two paths between $u$ and $v$ on $C$, one of odd length and the other of even length. There exists a unique vertex $w$ of $C$ at the same distance from $u$ and $v$ in $C$ and then in $G$ since $G$ is unicyclic which is the middle of the path of even length. The vertex $w$ is the unique vertex of $C$ that does not resolve the pair $(u,v)$. So any pair of vertices of $C$ resolves $C$.
\end{proof}

\begin{lemme}
\label{cycle3}
Let $G=(V,E)$ be a graph and $C$ be a cycle of $G$. If $V(C)=\{v_0,v_1,...,v_k\} $ and for any $i\leq j$ $d(v_i,v_j)=\min(j-i,k-j+i+1)$ \footnote{This condition  ensures that there is no shortcut between the vertices of $C$.}, then, for any set $S \subseteq C$ of size at least $3$, $S$ resolves $C$.
\end{lemme}
\begin{proof}
Let $S=\{v_a,v_b,v_c\}$ be any set of three vertices of $C$ and $v_x \ne v_y$ be two vertices of $C$. Assume by contradiction that $S$ does not resolve the pair $(v_x,v_y)$.

Note that neither $v_x$ nor $v_y$ belongs to $ \{v_a,v_b,v_c \}$ since otherwise $(v_x,v_y)$ would be resolved. 
Without loss of generality, we can assume that $v_x=v_0$ and $a <b < c$. 

Assume first that $y < a$. The shortest path on $G$ between $v_0$ and $v_a$ cannot contain $v_y$ otherwise $d(v_x,v_a) > d(v_y,v_a)$. Thus, $d(v_0,v_a)=k-a+1$ and similarly $d(v_0,v_b)=k-b+1$ so in particular $d(v_0,b) < d(v_0,a)$. Consider now the path between $v_y$ and $v_b$. If this path passes through $v_a$, then $d(v_y,a)<d(v_y,b)$ and if this path passes through $v_0$ then $d(v_y,b)>d(v_0,b)$. Both cases give a contradiction with the assumption that $S$ does not resolve the pair $(v_0,v_y)$.

Assume now that $a < y <b $. If the shortest path between $v_0$ and $v_b$ passes through $v_y$ then $d(v_0,v_b)>d(v_y,v_b)$ gives a contradiction. Thus, $d(v_0,v_b)=k-b+1$ and so $d(v_0,v_c)=k-c+1$. Similarly if the path between $v_y$ and $v_b$ passes through $v_0$ then $d(v_y,v_b)>d(v_0,v_b)$. So $d(v_y,v_b)=b-y$ and $d(v_y,v_c)=c-y$. We get $b-y=k-b+1$ and $c-y=k-c+1$ which is impossible since $b\neq c$.

The two last cases, $b < y <c $ and $y>c$, are respectively symmetric to the cases  $a < y <b $ and $y<a$.
\end{proof}

The following result has been stated in \cite{eroh2017comparison} but the proof contains a flaw. We provide a corrected version of the proof in Appendix~\ref{appendix}. It bounds the variation of the metric dimension when an edge is deleted in some conditions.

\begin{restatable}{lemme}{lemmevardim}
\label{lemme:vardim}
Let $G=(V,E)$ be a graph and $C$ be a cycle of $G$. Let $V(C)=\{v_0,v_1,...,v_k\}$ be the vertices of $C$. Denote by $G_i=(V_i,E_i)$ the connected components of the vertex $u_i$ in $G\setminus E(C)$. If, for every $i \ne j$, $V_i \cap V_j = \emptyset$, then, for any $e \in E(C)$, $\dim(G) \leq \dim(G-e)+1$.
\end{restatable}

The following lemma is a well-known fact about twins and resolving sets.

\begin{lemme}\label{twins}
Let $u$ and $v$ be two twins of a graph $G$. Any resolving set $S$ of~$G$ verifies $S \cap \{u,v\} \neq \emptyset$.
\end{lemme}

Lemma~\ref{dimbranch} is a crucial observation for studying resolving sets in particular in trees.

\begin{lemme}\label{dimbranch}
Let $G=(V,E)$ be a connected graph, $u$ be a vertex of $G$ and $S$ be a resolving set of~$G$. At most one connected component of $G \setminus \{u\}$ does not contain any vertex of $S$.
\end{lemme}
\begin{proof}
Assume by contradiction that two connected components $G_i$ and $G_j$ of $G \setminus \{x\}$ do not contain any vertex of $S$. Let $v \in V$ and $w \in V(G_j)$ be two vertices incident to $u$. Then, no vertex in $S$ resolves the pair $(v,w)$ since, for every $s \in S$, $d(s,v)=d(s,w)=d(s,u)+1$.
\end{proof}

\section{Bounds for the metric dimension}\label{section3}
\begin{defn}\label{defboundmin}
Let $G$ be a graph. Recall that $L(G)=\sigma(G)-ex(G)$. If $G$ is a path $P_n$ for some $n \geq 1$, let $L(G)=1$ (so $L(T)=\dim(T)$ for all trees).
\end{defn}

The goal is to prove Theorem~\ref{thm:mainbis} we recall here:
\mainbisstate*

This result implies the following one:
\corST*

\begin{proof}
For any graph $G=(V,E)$ that is not a tree and any edge $e \in E$ such that $G-e$ is connected, $L(G) \leq L(G-e)$. Indeed if a major vertex $v$ has terminal degree $d \geq 2$ in $G$, then $v$ is still a major vertex in $G-e$ of terminal degree at least $d$. So $L(G) \leq L(G-e)$ and then, for $T$ a spanning tree of $G$, $L(G) \leq L(T)$. As $\dim(G) \leq L(G) + 6c(G) $ and $\dim(T)=L(T)$ we get $\dim(G) \leq \dim(T)+6c(G)$.
\end{proof}

The rest of the section is devoted to prove Theorem~\ref{thm:mainbis}.

\subsection{Construction of the resolving set}

If $c(G)=0$ then $G$ is a tree and $\dim(G)=L(G)$ by Lemma \ref{dimtree}. If $c(G)=1$ let us prove a stronger result.

\begin{lemme}\label{lem:1cycle}
Let $G=(V,E)$ be a connected unicyclic graph. Then $\dim(G) \leq L(G)+3$.
\end{lemme}
\begin{proof}
Let $uv$ be an edge of the cycle. Let $T=G-e$, then $\dim(T)=L(T)$ by Lemma \ref{dimtree} and $\dim(G) \leq \dim(T)+1$ by Lemma \ref{lemme:vardim}. As $L(T) \leq L(G)+2$ we get the inequality.
\end{proof}

 We now focus on the case $c(G) \geq 2$.
The first part of the proof will consist in defining a subset $S$ of vertices. We then prove in the second part of the proof that it is, indeed, a resolving set.
In order to build this set $S$, we first find a small subset of vertices $M$ such that $G \setminus M$ is a forest and each connected component of $G \setminus M$ has at most two edges incident to $M$. We then construct the set $S$.

 Let us start with a simple lemma.

\begin{lemme}

Let $G$ be a connected graph with no vertex of degree $1$ that is not an induced cycle. There exists a feedback vertex set $X$ of size $\tau(G)$ containing only vertices of degrees at least $3$.

\end{lemme}

\begin{proof}
Let $X$ be a minimum feedback vertex set with the minimum number of vertices of degrees less than $3$. Note that $X$ does not contain vertices of degree $1$.
Assume by contradiction that $X$ contains a vertex $x$ of degree $2$. Let $P$ be the maximal path of vertices of degree $2$ containing $x$. Since $G$ is not a cycle (and is not acyclic otherwise $X$ would be empty), $P$ does not contain the whole graph. Let $y$ be an endpoint of $P$ adjacent to a vertex $z$ of $V \setminus P$. Let $X'=X \setminus \{x\} \cup \{z\}$. The set $X'$ is still a feedback vertex set, a contradiction with the minimality of $X$.  
\end{proof}

Let $X$ be a feedback vertex set of $G$ only containing vertices of degrees at least $3$ in the graph where all the vertices of degree $1$ have been iteratively removed\footnote{Note that we can assume that the resulting graph is not a cycle since otherwise the graph is unicyclic and the conclusion follows by Lemma~\ref{lem:1cycle}.}.
Let $G_1,G_2,...,G_k$ be the connected components of $G \setminus X$. Note that each $G_i$ is a tree. For each $G_i$, let $X_i \subseteq X$ be the set of vertices of $X$ connected to at least one vertex of $G_i$. Let $N_i \subseteq V(G_i)$ be the set of vertices in $G_i$ adjacent to (at least) one vertex of $X_i$.

Let $T_i$ be the minimal subtree of $G_i$ containing the vertices of $N_i$. In other words, $T_i$ is the subtree of $G_i$ restricted to the union of the paths between $a$ and $b$ for any pair $a,b \in N_i$. Let $T'_i$ be the tree built from $T_i$ by adding to each vertex $u \in N_i$, $|N(u)\cap X|$ pending degree $1$ vertices. Let $M_i$ be the set of vertices in $T'_i$ of degree at least $3$ and $M:=X \cup (\bigcup_{i=1}^k M_i)$. Figure \ref{fig:notation} illustrates these notations.

\begin{figure}[ht]
    \centering
    \includegraphics[scale=1]{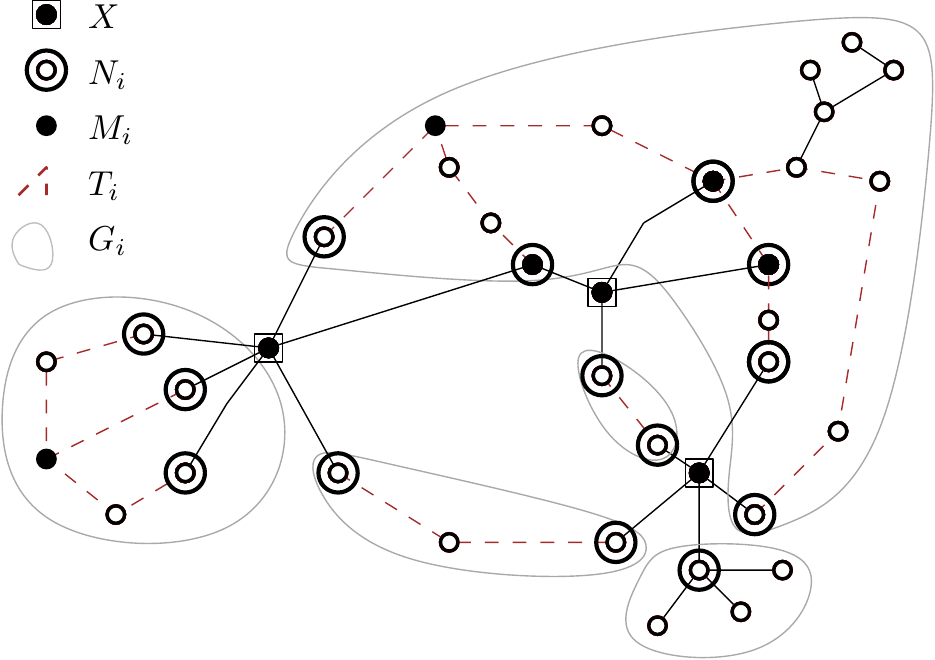}
    \caption{Illustration of the notations $X,N_i,M_i,T_i$ and $G_i$.}
    \label{fig:notation}
\end{figure}

\begin{lemme}\label{structure2}
For each connected component $H$ of $G \setminus M$ there are at most two edges in $G$ between $H$ and $G \setminus H$.
\end{lemme}
\begin{proof}
Let $U_H$ be the minimal subtree of $H$ containing all the vertices incident to an edge between $H$ and $G \setminus H$. Then, for each edge between a vertex $v \in H$ and a vertex in $G \setminus H$, add one new vertex in $U_H$ adjacent to $v$. Let us still denote by $U_H$ the resulting graph. Note that $U_H$ has as many degree $1$ vertices as edges leaving $H$. So, there are at most two edges with exactly one endpoint in $H$ if and only if $U_H$ has no vertex of degree three. The graph $U_H$ is isomorphic to a subgraph of $T'_i$, then, by definition of $T'_i$, $U_H$ does not contain vertex of degree three.

\end{proof}

Lemma~\ref{structure2} indeed implies the following.

\begin{cor} 
Every connected component of $G \setminus M$ is connected to at most two vertices of $M$.
\end{cor}

A connected component of $G\setminus M$ can be attached to $M$ in three different ways, called {\em Types}, illustrated in Figure~\ref{figexh2}. A connected component of $G \setminus M$ has \emph{Type $A$} (respectively \emph{Type $B$}) if there are exactly two edges between $H$ and $M$ with distinct endpoints in $H$ and such that their endpoints in $M$ are distinct (resp. the same).
A component $H$ has \emph{Type $C$} if all the edges of $G$ between $H$ and $M$ have the same endpoint in $H$ (but possibly distinct endpoints in $M$). 

Let $H$ be a connected component of $G \setminus M$ of Type $A$ or $B$ and let $x$ and $y$ be the two endpoints in $M$ of the edges between $H$ and $M$. Let $\rho_H$ be one vertex on the path in $H$ between $x$ and $y$ such that $|d_H(x,\rho_H)-d_H(y,\rho_H)| \leq 1$. In other words, $\rho_H$ is one of the vertices in the middle of the path between $x$ and $y$ in $H$. Let $P$ be the set of the vertices $\rho_H$ for all the connected components $H$ of \emph{Type} $A$ and $B$.
 
\begin{figure}[ht]
    \centering
    \includegraphics[scale=0.6]{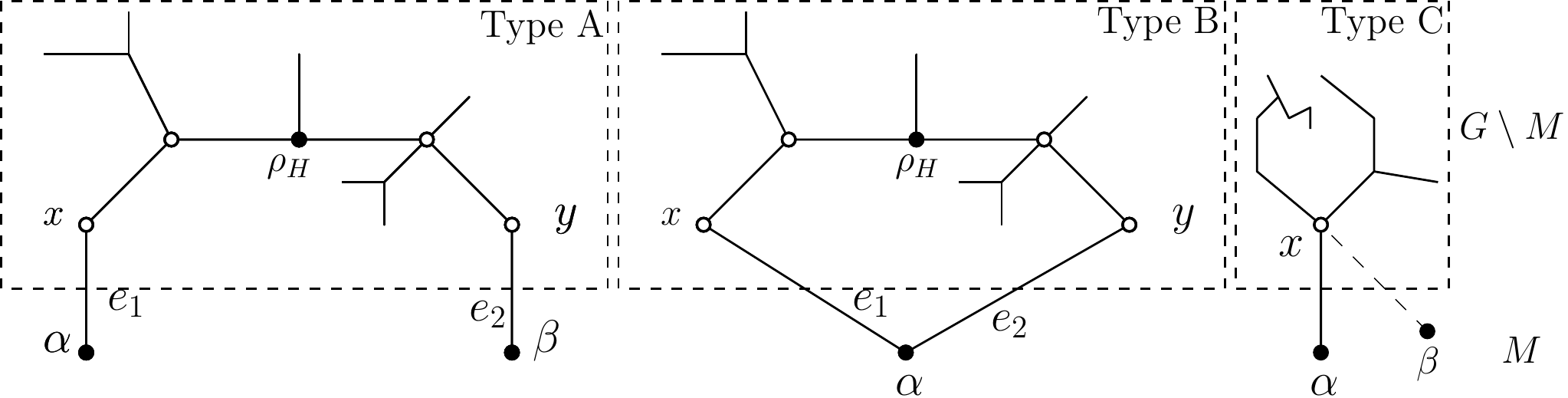}
    \caption{The three \emph{Types} of connected components of $G\setminus M$.  }
    \label{figexh2}
\end{figure}

\begin{lemme}\label{use_of_p2}
Let $H$ be a connected component of $G \setminus M$ of Type $A$ or $B$. Let $\rho$ be the vertex of $P \cap H$. Then, for any vertex $v \in H$, there is no shortest path between $v$ and $\rho$ using vertices in $G \setminus H$. Moreover, for every $z \in \partial H$, there is a shortest path between $z$ and $\rho$ using only vertices in $H \cup \{z\}$.
\end{lemme}

\begin{proof}
Let $x$ and $y$ be the two vertices in $H$ adjacent to a vertex of $M$, $\alpha$ be the vertex of $M$ adjacent to $x$ and $\beta$ the vertex in $M$ adjacent to $y$. By definition of \emph{Types} $A$ or $B$, $x \neq y$ (but $\alpha$ and $\beta$ could be the same vertex if the \emph{Type} is $B$).

By definition of $\rho$, we have $|d_H(x,\rho)-d_H(y,\rho)| \leq 1$. Let $v \in V(H)$. Assume by contradiction that a shortest path between $v$ and $\rho$ in $G$ passes through vertices in $G \setminus H$. So this shortest path between $\rho$ and $v$ passes through $x$ and $y$. 
By symmetry we can assume $d_G(v,\rho)=d_H(v,x)+d(x,\alpha)+d(\alpha,\beta)+d(\beta,y)+d_H(y,\rho)$. In other words, the path from $v$ to $\rho$ passes through $x$ and then $y$.

Assume by contradiction that $d_H(v,x)+d(x,\alpha)+d(\alpha,\beta)+d(\beta,y)+d_H(y,\rho) \leq d_G(v,\rho)$. As $x\alpha$ and $y\beta$ are edges, we have $d_H(v,x)+d(\alpha,\beta)+d_H(y,\rho) +2 \leq d_G(v,\rho)$. By triangular inequality, $d_G(v,\rho) \leq d_H(v,\rho) \leq d_H(v,x)+d_H(x,\rho)$. So $d_H(x,\rho) \geq d_H(y,\rho)+d(\alpha,\beta)+2$ which contradicts $|d_H(x,\rho)-d_H(y,\rho)| \leq 1$. \smallskip

If $z$ is in $\partial H$ and $\alpha=\beta$ then $z=\alpha=\beta$ and the result is immediate. Otherwise assume that $z=\alpha$. We want to contradict $d(\alpha,\beta)+d(\beta,\rho) < d(\alpha,\rho)$. As $\alpha$ is adjacent to $x$ and $\beta$ adjacent to $y$ the inequality is equivalent to $d(\alpha,\beta) + d(y,\rho) < d(x,\rho)$. As $\alpha \neq \beta$, $d(\alpha,\beta) \geq 1$. So $d(y,\rho) +2 \leq d(x,\rho)$ which contradicts $|d_H(x,\rho)-d_H(y,\rho)| \leq 1$.
\end{proof}

We need one more definition to define our resolving set.
\begin{defn}\label{defz}
Let $H$ be a connected component of $G \setminus M$ of Type $A$ or $B$ and let $x$ and $y$ be the two vertices of $H$ adjacent to $M$. Let $u$ be any vertex of $H$. The \emph{projection $z_u$ of $u$ (on the path between $x$ and $y$)} is the unique vertex in the path between $x$ and $y$ in $H$ at minimum distance to $u$. 
\end{defn}

We now have all the ingredients to define the set $S$ that will be a resolving set based on resolving sets of each connected component $H$ of $G \setminus M$. The union of the resolving sets of the different graphs is not a resolving set for $G$, so we will need to add a few vertices. 
Moreover, the size of the union of the resolving sets of the connected components of $G \setminus M$, is not bounded by $L(G)+c$. Some vertices have to be removed from these resolving sets.

Let $H$ be a connected component of $G \setminus M$. Let $S_H'$ be a metric basis for $H$ such that, for each major vertex of terminal degree at least $2$ in $H$, all but one of its terminal vertices are in $S_H'$.

To get the announced bound we divided the component of \emph{Type} $B$ in two parts. Let $H$ be a component of \emph{Type} $B$, $x$ and $y$ be the two vertices of $H$ adjacent to a vertex of $M$. 
A component has type \emph{Type} $B_1$ if $x$ and $y$ are terminal vertices or exterior degree-two vertices. If $x$ (resp. $y$) is a exterior degree-two vertex, $x$ (resp. $y$) lies on exactly one path between a terminal vertex and a major vertex. Then, we define this terminal vertex as the terminal vertex of $x$ (resp. $y$). If $x$ (resp. $y$) is a terminal vertex, then $x$ (resp. $y$) is it own terminal vertex.
A component has \emph{Type $B_2$} if it is a component of \emph{Type} $B$ and not a component of \emph{Type} $B_1$.

We define the set $S_H$ which is identical to $S_H'$ but with the following slight modifications:
\begin{itemize}
    \item $H$ has \emph{Type} $A$. For every $x \in H$ adjacent to $M$, if $x$ is a terminal vertex and its major vertex $v$ in $H$ has terminal degree at least $2$, then, if $x \in S_H'$, remove $x$ from $S_H'$, otherwise remove from $S_H'$ another terminal vertex of $x$.

    \item $H$ has \emph{Type} $B_1$, we can assume that $S_H'$ contains the two terminal vertices of $x$ and $y$. Indeed, by Lemma \ref{dimtree}, there exists a metric basis of $H$ containing these two vertices.

    \item $H$ has \emph{Type} $C$, let $x$ be the unique vertex of $H$ adjacent to $M$. 
    If $x$ is a terminal vertex and its major vertex $v$ in $H$ has terminal degree at least $2$, then, if $x \in S_H'$, remove $x$ from $S_H'$, otherwise remove from $S_H'$ another terminal vertex of $x$.

    \item $H$ has \emph{Type} $C$ and $H$ is a path with one extremity adjacent to $M$. Let $w$ be a vertex of $M$ adjacent to $H$. If there is only one component of \emph{Type} C attached to $w$ that is a path connected to $w$ by an endpoint of the path, let $S_H= \emptyset$. If there are several such components, then let $S_H = \emptyset$ for one of these components and $S_K$ be the extremity of the path not adjacent to $w$ for all the other such components $K$ (or the unique vertex of $H$ if $H$ is reduced to a single vertex).

\end{itemize}

The set $S$ is defined as $S= M \cup P \cup (\bigcup_i S_i)$. The goal will consist in proving that $S$ is a resolving $S$.

\subsection{The set $S$ is a resolving set}

We will prove several lemmas that restrict the components where pairs of unresolved vertices can belong to. Let us first prove that they must belong to the same connected component of $G \setminus M$.

\begin{lemme} \label{sameh}
Let $u$ and $v$ be two vertices of $G$. If $u,v$ are not resolved by $S$ then there exists a connected component of $H \in G \setminus M$ such that both $u,v$ belong to $H$.

\end{lemme}

\begin{proof}
First note that since $M \subseteq S$, $u,v \notin M$.
So there exist $H_u$ and $H_v$, connected components in $G \setminus M$, such that $u \in H_u$ and $v \in H_v$. Assume by contradiction that $H_u \ne H_v$. 
\begin{itemize}
    \item Assume $H_u$ is of \emph{Type} $A$ or $B$ and let $x\alpha$ and $y\beta$ be the two edges connecting $H_u$ to $M$ (with $x,y \in H_u$). Let $\rho$ be the vertex in $P \cap H_u$. Since $v \notin H_u$, the shortest path between $v$ and $\rho$ passes through $\alpha$ or $\beta$. Up to symmetry, $\alpha$ is on the shortest path between $v$ and $\rho$ so $d(v,\rho)=d(v,\alpha)+d(\alpha,\rho)$. As $\alpha$ and $\rho$ are in $S$ we also have $d(u,\rho)=d(u,\alpha)+d(\alpha,\rho)$, a contradiction with Lemma~\ref{use_of_p2}.
    \item Assume now that both $H_u$ and $H_v$ are of \emph{Type} $C$. Let $\alpha$ and $\beta$ be the vertices of $M$ connected to $H_u$ and $H_v$ respectively. If $\alpha \neq \beta$, then the shortest path between $u$ and $v$ contains two distinct vertices of $S$. Hence, by Lemma~\ref{technical} $u$ and $v$ are resolved. We assume now that $\alpha=\beta$. 
    
    Since all the components $H$ of \emph{Type} C attached to $\alpha$ but at most one contain a vertex of $S_H$, by construction, $H_u$ or $H_v$ contains a vertex of $S$. Without loss of generality, there exists $\gamma \in S \cap H_u$. If $u$ is on the path between $\gamma$ and $v$ then $d(\gamma,u) < d(\gamma,v)$.
    Otherwise, let $m_u$ in $H_u$ be at the intersection of the path between $u$ and $\gamma$ and between $\alpha$ and $\gamma$. Then, vertices $m_u$ and $\alpha$ are on the shortest path between $u$ and $v$. By Lemma~\ref{technical}, one of them must resolve $u$ and $v$. By assumption, it is not $\alpha$. If it is $m_u$, then we would have $d_G(u,\gamma) \neq d_G(v,\gamma)$. Since the shortest paths between $u$, $v$ and $\gamma$ go through $m_u$, we obtain a contradiction. 
    \end{itemize}

\end{proof}

We now prove that, if two vertices are in the same connected component of $G \setminus M$, then they are resolved by $S$. We start with connected components of \emph{Type} A.

\begin{lemme}\label{codez2}
Let $H$ be a connected component of $G \setminus M$ of Type A. Let $u$ and $v$ be two vertices of $H$ such that,for all $ s$ in $S$, $d(u,s)=d(v,s)$. Then, $z_u=z_v$.
\end{lemme}

\begin{proof}
Let $u\alpha$ and $v\beta$ be the two edges between $H$ and $M$ with $u \in H$ and $v \in H$.
The graph $H$ is a tree with a path between the two vertices $u$ and $v$. Assume $z_u \neq z_v$ and, without loss of generality, we can suppose that $z_u \neq \rho$ with $\rho$ the vertex of $P \cap H$.
We have $d(u,\alpha)=d(u,z_u)+d(z_u,\alpha)$, $d(v,\alpha)=d(v,z_v)+d(z_v,\alpha)$, $d(u,\beta)=d(u,z_u)+d(z_u,\beta)$ and $d(v,\beta)=d(v,z_v)+d(z_v,\beta)$. As $d(u,\alpha)=d(v,\alpha)$ and $d(u,\beta)=d(v,\beta)$ we get \[d(z_u,\alpha)+d(z_v,\beta)=d(z_u,\beta)+d(z_v,\alpha).\]
The vertices $z_u$ and $\rho$ are distinct and both between $\alpha$ and $\beta$. So $z_u$ is between $\alpha$ and $\rho$ or $\beta$ and $\rho$. Assume $z_u$ is between $\alpha$ and $\rho$. Then $d(z_u,\alpha) \leq d(z_u,\beta)$, so $d(z_v,\alpha) \leq d(z_v,\beta)$ meaning $z_v$ is also between $\alpha$ and $\rho$. The shortest path between $\alpha$ and $\rho$ passes through $z_u$ and $z_v$ by Lemma~\ref{use_of_p2}. Assume $z_u$ is closer than $z_v$ to $\rho$. Then $d(\alpha,\rho)=d(\alpha,z_v)+d(z_v,z_u)+d(z_u,\rho)$ gives 
\[d(\alpha,z_u)+d(z_v,\rho)=d(\alpha,z_v)+d(z_u,\rho)+2d(z_u,z_v).\]
Use now the paths to $\rho$: $d(u,\rho)=d(u,z_u)+d(z_u,\rho)$ as $z_u \neq \rho$ and $d(v,\rho) \leq d(v,z_v)+d(z_v,\rho)$. Then $d(u,\alpha)=d(u,z_u)+d(z_u,\alpha)$, $d(v,\alpha)=d(v,z_v)+d(z_v,\alpha)$ gives
\[d(z_v,\alpha)+d(z_u,\rho) \leq d(z_u,\alpha)+d(z_v,\rho).\]
A combination of the previous equality gives $d(z_u,z_v) \leq 0$ so $z_u=z_v$. 
\end{proof}

\begin{lemme} \label{typea}
Let $H$ be a connected component of $G \setminus M$ of Type A. Let $u$ and $v$ be two vertices of $H$ such that, for all $s$ in $S$, $d(u,s)=d(v,s)$. Then $u=v$.
\end{lemme}
\begin{proof}
Assume by contradiction $u \neq v$.
Let $\alpha$ and $\beta$ be the two vertices of $M$ adjacent to $H$. By construction of $S_H$, $S_H \cup \{\alpha,\beta\}$ is a resolving set of $H \cup \{\alpha,\beta\}$. Let $\gamma$ which resolves the pair $(u,v)$ in $H \cup \{\alpha,\beta\}$. By Lemma \ref{codez2}, $z_u=z_v$ so $\gamma$ still resolves $(u,v)$ in $G$. Indeed if $z_\gamma=z_u$ then the distances are the same in $G$ and in $H \cup \{\alpha,\beta\}$. If $z_\gamma \neq z_u$ then $d_G(u,\gamma)=d_H(u,z_u)+d_G(z_u,\gamma)$ and $d_G(v,\gamma)=d_H(v,z_v)+d_G(z_v,\gamma)$. As $d_H(u,\gamma) \neq d_H(v,\gamma)$ with $d_H(u,\gamma)=d_H(u,z_u)+d_H(z_u,\gamma)$ and $d_H(v,\gamma)=d_H(v,z_u)+d_H(z_u,\gamma)$ we get $d_H(u,z_u) \neq d_H(v,z_u)$ so $d_G(u, \gamma) \neq d_G(v,\gamma)$. So $\gamma$ resolves $(u,v)$ in $G$, a contradiction.

\end{proof}

\begin{lemme}\label{codez}
Let $H$ be a connected component of $G \setminus M$ of Type B. If $u,v \in H$ are not resolved by $S$, then $z_u=z_v$.
\end{lemme}

\begin{proof}
Let us prove it by contradiction. Let $\alpha$ be the vertex of $M$ connected to $H$. Let $x,y$ be the two vertices of $H$ connected to $\alpha$. \smallskip

\noindent
\textbf{Case 1: $H$ has \emph{Type} $B_1$} \\ 
 By construction, $\{\alpha,\rho,y\} \subseteq S$ with $y$ such that $z_y$ is connected to $\alpha$. Assume by contradiction $z_u \neq z_v$. We first show that $(z_u,z_v)$ is resolved by $\{\alpha,\rho,y\}$. Indeed, $d(y,z_u)=d(y,z_y)+d(z_y,z_u)$ and $d(y,z_v)=d(y,z_y)+d(z_y,z_v)$. Lemma \ref{cycle3} ensures that $(z_u,z_v)$ is resolved by a vertex of $\{\alpha,\rho,z_y\}$ and if $z_y$ resolves $(z_u,z_v)$, then $y$ resolves $(z_u,z_v)$. So $(z_u,z_v)$ is resolved by a vertex of $\{\alpha,\rho,y\}$, let $\gamma$ be such a vertex.
     
 If $z_v=\rho$, then $d(\gamma,u)=d(\gamma,z_u)+d(z_u,u)$ and $d(\gamma,v)=d(\gamma,\rho)+d(\rho,v)$, so $d(z_u,u) \neq d(\rho,v)$. As $\rho \in S$, $d(u,\rho)=d(v,\rho)$ so $d(z_u,u)<d(\rho,v)$. We exploit now the equalities $d(\alpha,u)=d(\alpha,v)$ and $d(\alpha,u)=d(\alpha,z_u)+d(z_u,u)$. By definition of $\rho$, $d(\alpha,z_u) \leq d(\alpha,\rho)$ and $d(z_u,u)<d(\rho,v)$. So $d(\alpha,v)=d(\alpha,\rho)+d(\rho,v)>d(\alpha,u)$, a contradiction.

If $z_v \neq \rho$, then $d(\gamma,u)=d(\gamma,z_u)+d(u,z_u)$ and $d(\gamma,v)=d(\gamma,z_v)+d(v,z_v)$. By hypothesis $d(\gamma,u)=d(\gamma,v)$, so $d(u,z_u) \neq d(v,z_v).$ We can assume by symmetry $d(z_u,u)<d(z_v,v)$. Let $\beta \in \{\alpha,\rho\}$, such that $d(z_u,\beta) \leq d(z_v,\beta)$.
Such a vertex exists since the distances $d(\alpha,z_u)+d(z_u,\rho)$ and $d(\alpha,z_v)+d(z_v,\rho)$ are the same if $z_u,z_v$ are both on the same side of the $xy$-path with respect to $\rho$ and differ by at most one otherwise.
Then, $d(\beta,u)=d(\beta,v)$ and $d(\beta,u)=d(\beta,z_u)+d(z_u,u)$. But $d(\beta,z_u) \leq d(\beta,z_v)$ and $d(z_u,u)<d(z_v,v)$. So $d(\beta,v)=d(\beta,z_v)+d(z_v,v)>d(\beta,u)$, a contradiction. 
\smallskip

\noindent
\textbf{Case 2: $H$ has  \emph{Type} $B_2$} \\
    As $S$ contains $S_H$ which is a resolving set of $H$, there exists $\gamma \in H$ such that $d_H(u,\gamma) \neq d_H(v,\gamma)$. By hypothesis $d_G(u,\gamma) = d_G(v,\gamma)$.
    
    Assume first $z_{\gamma} = \rho$. Let us prove that $d_G(u,\gamma)=d_H(u,\gamma)$ and $d_G(v,\gamma)=d_H(v,\gamma)$, which gives a contradiction. By symmetry it is enough to prove that $d_G(u,\gamma)=d_H(u,\gamma)$.
    If $z_u=z_{\gamma}=\rho$ then $d_G(u,\gamma)=d_G(u,\rho)+d_G(\rho,\gamma)=d_H(u,\rho)+d_H(\rho,\gamma)=d_H(u,\gamma)$ and the conclusion follows. 
    If $z_u \neq z_{\gamma}$, then
    
    \[d_G(u,\gamma)=d_G(u,z_u)+d_G(z_u,\rho)+d_G(\rho,\gamma).\]
    
    By Lemma~\ref{use_of_p2}, $d_G(z_u,\rho)=d_H(z_u,\rho)$. We have $d_G(u,z_u)=d_H(u,z_u)$ and $d_G(\rho,\gamma)=d_H(\rho,\gamma)$ since the paths between these vertices are unique. So $d_G(u,\gamma)=d_H(u,\gamma)$.
    
    So, from now on, we can assume that $z_{\gamma} \neq \rho$. Since $\{\rho,\alpha\}$ does not resolve $(u,v)$, we have $d(u,z_u)+d(z_u,\rho)=d(v,z_v)+d(z_v,\rho)$, and $d(u,z_u)+d(z_u,\alpha)=d(v,z_v)+d(z_v,\alpha)$. Thus,
    \[d(z_u,\rho)+d(z_v,\alpha)=d(z_v,\rho)+d(z_u,\alpha).\] 
    Since $\rho$ and $\alpha$ are almost opposed on the smallest cycle containing them, we also have
    \[d(z_u,\rho)+d(z_u,\alpha)=d(z_v,\rho)+d(z_v,\alpha)+ \epsilon\] 
    with $\epsilon \in \{-1,0,1\}$. Summing the two equalities gives $2d(z_u,\rho)=2d(z_v,\rho)+ \epsilon$. So by parity $\epsilon=0$. Then, $d(z_u,\rho)=d(z_v,\rho)$ and finally $d(z_u,\alpha)=d(z_v,\alpha)$. Since  $d(u,\alpha)=d(v,\alpha)$, we obtain $d(u,z_u)=d(v,z_v)$. 
    
    If $z_\gamma \notin \{z_u,z_v\}$, then $z_u$ and $z_v$ are at the same distance to $\alpha$, $\rho$ and $z_\gamma$, so, by Lemma \ref{cycle3}, $z_u=z_v$ . If $z_\gamma \in \{z_u,z_v\}$, up to symmetry we can assume $z_\gamma = z_u$. Then $d(u,\gamma) \leq d(u,z_u)+d(z_u,\gamma)$ and $d(v,\gamma)=d(v,z_v)+d(z_v,z_u)+d(z_u,\gamma)$. As $d(u,\gamma)=d(v,\gamma)$ and $d(u,z_u)=d(v,z_v)$ we get $d(z_v,z_u) \leq 0$ so $z_u=z_v$. 
\end{proof}

\begin{lemme} \label{typeb}
Let $H$ be a connected component of $G \setminus M$ of Type B.
The set $S$ resolves all the pairs of vertices of $H$.

\end{lemme}

\begin{proof}
Let $u,v$ be two vertices of $H$ which are not resolved by $S$.
By Lemma \ref{codez}, $z_u=z_v$. Let $z =z_u=z_v$, if $\deg(z)=2$ then $u=v=z$ and the result is proven. We can assume from now on that $\deg(z) \geq 3$. 
\smallskip

\noindent
\textbf{Case 1: There exists $\gamma \in S_H$ which resolves the pair $(u,v)$ in $H$.} 

If $z_{\gamma}=z_u=z_v$ then the distances between $u$ (resp. $v$) and $\gamma$ are the same in $H$ and $G$, a contradiction.

So we can assume that $z_{\gamma} \neq z_u$. We have $d_{H}(\gamma,u)=d_{H}(u,z)+d_{H}(z,\gamma)$ and $d_{H}(\gamma,v)=d_{H}(v,z)+d_{H}(z,\gamma)$. Since $d_{H}(\gamma,u) \neq d_{H}(\gamma,v)$, we have $d_{H}(u,z) \neq d_{H}(v,z)$. 
Now, since by Lemma~\ref{use_of_p2}, for $w \in \{ u,v \}$, $d_H(w,\rho)=d_G(w,\rho)$ and $d(w,\rho)=d(w,z)+d(z,\rho)$, $\rho$ resolves $(u,v)$, a contradiction.
\smallskip

\noindent
\textbf{Case 2: The pair $(u,v)$ is not resolved by $S_H$ in $H$.}

This case can only happen if $H$ has \emph{Type} $B_1$ (since otherwise no vertex of $S_H'$ is removed). Then there exists a vertex $x$ such that $z_x$ is adjacent to $\alpha$ which resolves the pair $(u,v)$ in $H$. If $z=z_x$, then, $d(z,u)=d(z,x)-d(u,x)$ and $d(z,v)=d(z,x)-d(v,x)$ so $d(z,u) \neq d(z,v)$. If $z \neq x$ then $d(x,u)=d(x,z)+d(z,u)$ and $d(x,v)=d(x,z)+d(z,v)$. So $d(z,u) \neq d(z,v)$ in both cases. Hence $\alpha$ resolves the pair $(u,v)$ in $G$. As $z \neq \alpha$, $d(\alpha,u)=d(\alpha,z)+d(z,u) \neq d(\alpha,z)+d(z,v) = d(\alpha,v)$.
\end{proof}

\begin{lemme} \label{typec}
Let $H$ be a connected component of $G \setminus M$ of Type C. Then $S$ resolves any pair of vertices in $H$.

\end{lemme}
\begin{proof}
Assume by contradiction two vertices $u,v \in H$ with $u \neq v$ are not resolved by $S$. Let $x$ be the unique vertex of $H$ adjacent to $M$ and $m$ be a vertex of $M$ adjacent to $x$. Let $H'$ be the subgraph of $G$ with vertex set $V(H) \cup \{m\}$.

If $S$ contains a resolving set of $H$, then, since $x$ is a cut-vertex, $S$ resolves the pair $(u,v)$. So we can assume that at least one vertex of $S_{H}'$ has been removed during the construction of $S$. 

Note that since $m$ does not resolve $(u,v)$, $d(u,x)=d(v,x)$ in $H$. So in particular $H$ cannot be a path with endpoint $x$. So by construction of $S$, we can assume that $x$ is a terminal vertex in $H$ and its major vertex has terminal degree at least two in $H$.

By construction of $S_H$, $S \cup \{x\}$ is a resolving set for $H$. Then $S \cup \{m\}$ is a resolving set for $H'$. Let $\gamma \in S_H \cup \{m\}$ that resolves $u$ and $v$ in $H'$. The distances between $u$, $v$ and $\gamma$ in $G$ and $H'$ are the same so $\gamma$ resolves $u$ and $v$, a contradiction.
\end{proof}

\begin{lemme}\label{resolving}
The set $S$ is a resolving set of $G$.
\end{lemme}

\begin{proof}
Let $(u,v)$ be a pair of vertices that is not resolved by $S$. Assume by contradiction $u \neq v$. By Lemma \ref{sameh}, there exists a connected component $H$ of $G \setminus M$ such that $u \in H$ and $v \in H$. Then, if $H$ has \emph{Type} A, by Lemma \ref{typea}, $u=v$. If $H$ has \emph{Type} B, by Lemma \ref{typeb}, $u=v$. If $H$ has \emph{Type} C then, by Lemma \ref{typec}, $u=v$.
\end{proof}

\subsection{Upper bound on the size of $S$}

Lemma \ref{resolving} ensures that $\dim(G) \leq |S|$. So Theorem~\ref{thm:mainbis} holds if $|S| \le L(G)+6c$.
The set $S$ is a union of three sets that we will bound the size separately. We use the following result on minors to get the bounds.
 
Let $G$ be a multigraph. The graph $H$ is a {\em minor} of $G$ if $H$ can be obtained from $G$ via a sequence of edge deletions, vertex deletions and edge contractions (the edge contraction operation can create parallel edges between two vertices or loops). One can easily check that the minor operation can only decrease the cycle rank.

\begin{lemme}\label{boundMalt2}
$|M| \leq 2c(G)-2$. 
\end{lemme}
\begin{proof}
If $|M|\leq 2$ then the inequality holds since we can assume $c \geq 2$. So we can assume that $|M| \ge 3$. Let $K$ be the multigraph (with possible loops) with vertex set $M$ where, for each connected component $H$ of $G \setminus M$ of \emph{Type} $A$ or $B$ with endpoints $x$ and $y$ in $M$ (that might be identical), we create an edge between $x$ and $y$ in $K$. Note that $K$ is a minor of $G$ as it can be obtained from $G$ by contracting edges in components $H$ of $G \setminus M$ into a single edge.

Every vertex in $K$ has degree at least $3$. The vertices of $M\setminus X$ have degree at least $3$ by definition of $M_i$ for every $i$. 
By construction of $X$, $x \in X$ has degree at least $3$ in the graph starting from $G$ and removing the degree one vertices. Three adjacent edges belong to cycles so contribute to the degree of $x$ in $K$ so $\deg_{K}(x) \geq 3$. 

We have $3 |V(K)| \le \sum_{v \in V(K)} \deg(v) = 2 |E(K)|$. So $3|M| \leq 2 |E|$. Since the cycle rank of $K$ is at most $c$, so $|E| \leq c + |M| - 1 $. A combination of these inequalities gives $|M| \leq 2c(G)-2$.
\end{proof}

\begin{lemme}\label{boundP2}
$|P| \leq c + |M|-1$.
\end{lemme}
\begin{proof}
Let $K'$ be the multigraph (with loops) with vertex set $M$ and an edge between two vertices $x$ and $y$ if and only if there exists in $G \setminus M$ a connected component $H$ adjacent to $x$ and $y$. The graph $K'$ is a minor of $G$ since $K'$ can be obtained from $G$ by contracting edges with exactly one endpoint in $M$ until no such edge exists. One can easily notice that in $K'$ there is an edge $xy$ with multiplicity $k$ if and only if in $G \setminus M$ there are $k$ connected components attached to $x$ and $y$. 
Since $K'$ is a minor of $G$, $c(K') \leq c(G)$. As $K'$ contains $|M|$ vertices, $K'$ has at most $c+|M|-1$ edges. Thus, $G \setminus M$ has at most $c+|M|-1$ components of \emph{Type} A or B. Since $P$ contains one vertex in each component of \emph{Type} A or B, we have $|P| \leq c+|M|-1$.
\end{proof}

\begin{lemme}\label{sizecomp}
\[ \big| \bigcup_{H \text{con. comp. of } G \setminus M} S_H \big| \leq L(G) +c. \]
\end{lemme}
\begin{proof}
For every connected component $H$ of $G \setminus M$, let $\ell_H=\sum_{x \in H} (\ter(x)-1)$ over all the major vertices $x$ in $H$. 
We consider the three types of components.

Let $H$ be a connected component of \emph{Type} $A$. We claim that $|S_H|=\ell_H$. Indeed let $\epsilon \in \{0,1,2\}$ be the number of vertices of $H$ adjacent to $M$ which are terminal vertices connected to a major vertex of degree at least $2$. By construction $L(H)-\epsilon=|S_H|=\dim(H)- \epsilon$.

Let $H$ be a connected component of \emph{Type} $B_1$. Then $L(H)=|S_H|+1$ and $\ell_H=L(H)-2$ by construction so $|S_H| \leq \ell_H +1$. \\
Let $H$ be a connected component of \emph{Type} $B$ not $B_1$. Then, $L(H)=|S_H|$ and $\ell_H \geq L(H)-1$ because $H$ has not \emph{Type} $B_1$ so $|S_H| \leq \ell_H +1$. \\
Let $H$ be a component of \emph{Type} $C$ not a path. By construction of $S_H$, $|S_H|=\ell_H$. Indeed, let $\epsilon \in \{0,1\}$ be the number of vertices of $H$ adjacent to $M$ which are terminal vertices connected to a major vertex of degree at least $2$. By construction $|S_H|=\dim(H)- \epsilon=L(H)-\epsilon$ and $\ell_H = L(H)-\epsilon$. \\
Let $H$ be a component of \emph{Type} $C$ with $H$ a path. If $H$ is connected to $M$ by a vertex $x$ which is not an extremity of $H$ then $m$ is a major vertex of terminal degree 2 in $G$. So $|S_H|=\ell_H=1$. If $H$ is a path connected to a vertex $m \in M$ by an extremity, let $k \in \mathbb{N}$ be the number of such components connected to $m$. Denote them by $H_1,H_2,...,H_k$. If $k\geq 2$ then $m$ is a major vertex in $G$ with terminal degree $k$ and $|\cup_{i \leq k} S_{H_i}| =k-1=|L(G) \cap \{m\} \cup (\cup_{i \leq k} H_i)|$. If $k=1$ then $S_H= \emptyset$ so $\ell_H =0$. \\
There are at most $c(G)$ components of \emph{Type} $B$: for each component $H$ of \emph{Type} $B$ we can found a cycle in $G$ by adding the vertex of $M$ adjacent to $H$. By definition of $c(G)$, this gives at most $c(G)$ components of \emph{Type} $B$. Summing the inequalities gives the result:
\[
| \cup S_H | = \sum_{\text{\emph{Type} $A$}} |S_H| + \sum_{\text{\emph{Type} $B$}} |S_H| +\sum_{\text{\emph{Type} $C$}} |S_H| \] \[ \leq
\sum_{\text{\emph{Type} $A$}} \ell_H + \sum_{\text{\emph{Type} $B$}} (\ell_H+1) +\sum_{\text{\emph{Type} $C$}} \ell_H
\leq
L(G)+c(G)
\]
\end{proof}

Finally, we can prove Theorem~\ref{thm:mainbis}:

\begin{proof} 
The set $S$ is a resolving set so $\dim(G) \leq |S|$. By definition $S= M \cup P \cup (\bigcup_H S_H)$. By Lemma \ref{boundMalt2}, $|M| \leq 2c(G)$, by Lemma \ref{boundP2}, $|P| \leq 3c$ and by Lemma \ref{sizecomp}, $| \cup S_H | \leq L(G) +c(G)$. Summing the inequalities give $\dim(G) \leq L(G) + 6c(G)$.
\end{proof}

\begin{figure}[ht]
    \centering
    \includegraphics[scale=1.1]{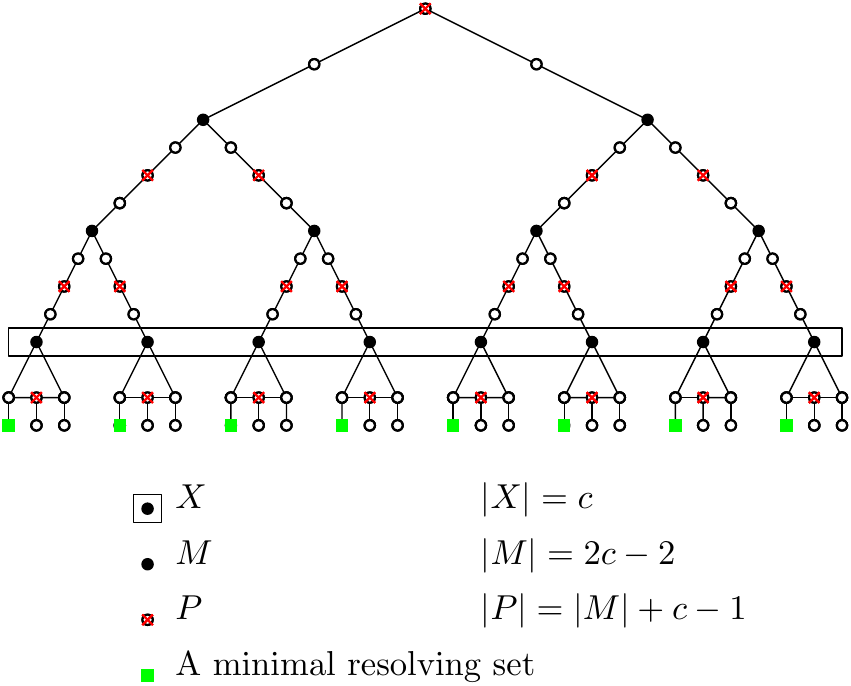}
    \caption{Tightness of Lemmas \ref{boundMalt2}, \ref{boundP2} and \ref{sizecomp}.}
    \label{figtightlemme}
\end{figure}

One can naturally ask if this upper bound is optimal.
Figure \ref{figtightlemme} gives an example of graph where Lemmas~\ref{boundMalt2},~\ref{boundP2} and~\ref{sizecomp} are tight. It ensures that our analysis of the construction is optimal but not necessarily the construction itself. Indeed, the metric dimension of the graph of Figure~\ref{figtightlemme} is $8$ and the square vertices form a metric basis.

\section{Metric dimension and zero forcing sets}\label{section4}

In this section, we study how the metric dimension and the zero forcing number can be modified when an edge is added to a graph.
Then we prove a weakening of Conjecture \ref{conj}, as a consequence of Theorem~\ref{thm:mainbis}.
We then give a short proof of Conjecture~\ref{conj} for unicyclic graphs. We will then generalize this result to prove the conjecture for cactus graphs. We finally prove a strengthening of Conjecture~\ref{conj} when the graph is unicyclic and the unique cycle has odd length.

\subsection{Edge modifications and consequences for Conjecture~\ref{conj}}\label{sec:consZF}

The following lemma ensures that the variations of the zero forcing number when an edge is added or deleted an edge are small \cite{EDHOLM2012}. 

\begin{lemme}\cite{EDHOLM2012}\label{varZ}
Let $G=(V,E)$ be a graph and $e \in E(G)$, then $Z(G)-1 \leq Z(G-e) \leq Z(G)+1$.
\end{lemme}

We have a more precise result if $Z(G+e) < Z(G)$ which will be useful later.
\begin{lemme}
\label{rev}
Let $G=(V,E)$ be a graph and $u$ and $v$ two vertices of $G$ such that $e=uv \notin E$. If $Z(G+e) < Z(G)$, then for any minimum zero forcing set of $G+e$, at some step $u$ forces $v$ or $v$ forces $u$. 
\end{lemme}
\begin{proof}
     By contradiction if a zero forcing set of minimal size for $G+e$ does not use the edge $e$ then it is a zero forcing set of $G$ so $Z(G) \leq Z(G+e)$.
\end{proof}

A similar statement does not hold for the metric dimension. However, Lemma \ref{lemme:vardim} gives some conditions where a similar result holds. Using these results we can get inequalities between the metric dimension and the zero forcing number for some classes of graphs.

\begin{cor}\label{ineq2}
Let $G$ be a connected unicyclic graph and $e$ be an edge such that $T=G-e$ is a tree. Then, $\dim(G) \leq Z(G)+2$.
\end{cor}
\begin{proof}
By Lemma~\ref{treecomp}, $\dim(T) \leq Z(T)$. Lemmas~\ref{lemme:vardim} and~\ref{varZ} ensure that $\dim(G) \leq Z(G)+2$.
\end{proof}

Eroh et al. \cite{eroh2017comparison} proved Conjecture~\ref{conj} for unicyclic graphs via a very long case analysis. They start from a tree $T$ achieving $\dim(T)=Z(T)$ and make a complete study of all the places where an edge could be added. We drastically simplify their proof by starting from a unicyclic graph $G$ and delete a well-chosen edge.

\begin{lemme}
\label{zero}
Let $G=(V,E)$ be a graph which is not a tree and $C \subseteq V$ a cycle of $G$. Then, there exists an edge $e \in E(C)$ such that $Z(G-e) \leq Z(G)$.
\end{lemme}

\begin{proof}
Let $Z \subseteq V$ be a minimum zero forcing set of $G$. Let $F \subseteq E$ be the \emph{forcing edges} in a sequence starting from $Z$, i.e., $uv \in F$ if and only if at some stage $u$ forces $v$ or $v$ forces $u$. 

We claim that at least one edge of $C$ is not in $F$. Indeed, if $u$ forces $v$ then $u$ is turned black before $v$. So, the first vertex $w$ of the cycle that is turned black cannot be turned black because of an edge of $C$ (such a vertex can already be black at the beginning of the proceed). Let $w_1,w_2$ be the two neighbors of $w$ on $C$. The vertex $w$ can force at most one of its two neighbors. So, without loss of generality, $w_2$ is not forced by $w$ and is turned black after $w$. So, if we remove the edge $e=ww_2$, $Z$ is still a forcing set of $G-e$ with the same sequence of applications of the color change rule that turned $G$ into black. Therefore, $Z(G-e) \leq Z(G)$.
\end{proof}

We obtain as a corollary the main result of~\cite{eroh2017comparison}.
\begin{cor}\label{unicycle}
Let $G$ be a unicyclic graph. Then, $\dim(G) \leq Z(G)+1$.
\end{cor}
\begin{proof}
Let $e$ be an edge of $C$ such that $Z(G-e) \leq Z(G)$. Such an edge exists by Lemma~\ref{zero} and according to Lemma~\ref{lemme:vardim}, $\dim(G) \leq \dim(G-e)+1$. Moreover, by Lemma~\ref{treecomp},  $\dim(G-e) \leq Z(G-e)$ since $G-e$ is a tree. The combination of these three inequalities gives $\dim(G) \leq Z(G)+1$.
\end{proof}
 
All these results together with Theorem~\ref{thm:mainbis} allow us to prove a weakening of Conjecture~\ref{conj} which we restate as follows:

\thmZFdim*

\begin{proof}
By induction with Lemma~\ref{zero}, there exists a spanning tree $T$ such that $Z(T) \leq Z(G)$. By Lemma \ref{treecomp}, $\dim(T) \leq Z(T)$ and thus, $\dim(G) \leq Z(G)+6c(G)$ by Theorem \ref{thm:mainbis}.
\end{proof}

We generalize the proof of Conjecture~\ref{conj} for unicyclic graphs to cactus graphs. Almost the same techniques can be applied. First, we define the cactus graphs class.
\begin{defn}
Any graph $G$ is a cactus graph if any edge $e \in E(G)$ is part of at most one cycle of~$G$.
\end{defn}
\begin{theorem}\label{edgedisjoint}
Let $G=(V,E)$ be a cactus graph. Then, $\dim(G) \leq c(G) +Z(G)$.
\end{theorem}

\begin{proof}
We prove this inequality by induction on $c(G)$. If $c(G)=0$, then $G$ is a tree and $\dim(G) \leq Z(G)$. If $c(G) >0$ let $C$ be a cycle of $G$. By Lemma~\ref{zero}, there exists an edge $e \in C$ such that $Z(G-e) \leq Z(G)$. By induction $\dim(G-e) \leq Z(G-e) +c(G-e) \leq Z(G) + c(G-e) = Z(G) + c(G)-1$ since $c(G-e)=c(G)-1$ for any $e \in C$. \\
To conclude, let us prove that Lemma~\ref{lemme:vardim} can be applied. Let $v_1,v_2...,v_k$ be the vertices of $C$. Let $G_i$ be the connected component of $v_i$ in $G \setminus C$. Assume by contradiction that two subgraphs $G_i$ and $G_j$ with $1 \leq i < j \leq k$ are not disjoint. Then, there exists a path $P$ between the vertices $v_i$ and $v_j$ in $G \setminus C$. Then $G$ contains two cycles with common edges: $C$ and a cycle containing $P$ and a path in $C$ between $v_i$ and $v_j$, a contradiction. So, by Lemma~\ref{lemme:vardim}, $\dim(G) \leq \dim(G-e)+1$ and then $\dim(G) \leq c(G) +Z(G)$.
\end{proof}

As a corollary, we obtain that cactus graphs satisfy Conjecture~\ref{conj}. It improves a result of~\cite{eroh2017comparison} that ensures that $Z(G) \le \dim(G)+2c(G)$ if $G$ has no even cycles. If $G$ has no even cycles, all its cycles are edge disjoint. Indeed, if two odd cycles share at least one edge then $G$ contains an even cycle.

\subsection{Unicyclic graphs with an odd cycle}

In this section, we consider the case where $G$ is unicyclic and its cycle has odd length. In this case, we will improve the inequality of Corollary~\ref{unicycle} to get $\dim(G) \leq Z(G)$. Such a result cannot be extended to $G$ with an even cycle, see Figure~\ref{figzdim} for an example. The intuitive reason why there is a difference between odd and even cycles is that, by Lemma~\ref{cycle1}, any pair of vertices resolves an odd cycle while it is far from being true for even cycles.

Before proving the main result of this section, we need some technical lemmas. Let $k \geq 1$ and let $G=(V,E)$ be a graph containing a unique cycle $C$ of length $2k+1$. For $u\in C$, let $T_u$ be the connected component of $u$ in $G'=(V,E\setminus E(C))$ rooted in $u$. Note that $T_u$ is a tree.  We call $u$ the root of $T_u$.
We say that $T_u$ is \emph{trivial} if $T_u=\{u\}$, is a \emph{rooted path} if $T_u$ is a path with $u$ at one extremity and is a \emph{rooted tree} otherwise. Note that a rooted path can be trivial (otherwise specified).
For $u \in C$, if $T_u$ is not trivial we denote by $\ell_u$ the terminal degree of $u$ in $G$. Else we let $\ell_u=0$.

\begin{lemme}
\label{deg1}
Let $G$ be an odd unicyclic graph. If there exists a vertex $u \in C$ such that $\ell_u \geq 1$, then there is an edge $e \in E(C)$ incident to $u$ such that $Z(G-e) \leq Z(G)$. 
\end{lemme}

\begin{proof}
By Lemma~\ref{rev}, it suffices to find a minimum zero forcing set for $G$ which does not use one of the two edges in $E(C)$ incident to $u$. Let $Z$ be a minimum zero forcing set of $G$. If $u\in Z$ then $u$ can force only one vertex and the result is proved since at most one edge incident to $u$ is used. We can assume $u \notin Z$. Let $P$ be an internal degree-two path between $u$ and a terminal vertex $l$ of $G$ which exists since $\ell_u \geq 1$. If there is a vertex $x$ in $P \cap Z$, then $(Z\setminus \{x\}) \cup \{l\}$ is still a minimum forcing set of $G$. Then, $l$ iteratively forces the vertices of $P$ until $u$ and we are back to the previous case. Finally, if $P$ and $u$ are initially white, then $u$ is the first vertex of $P \cup \{u\}$ which is turned black (eventually by one edge in $E(C)$). It then turns in black $P$. We cannot use the second edge of $E(C)$ since every vertex forces at most one vertex.
\end{proof}

\begin{lemme}
\label{cycle}
Let $G$ be an odd unicyclic graph.
Let $S\subseteq V(G)$ be such that for any $u$ on the cycle, $S$ is not a subset of $T_u$. Then, $S$ resolves $C$.
\end{lemme}
\begin{proof}
Let $\alpha$ and $\beta$ in $S$ such that $\alpha \in T_u$ and $\beta \in T_v$ with $u \neq v$. Assume by contradiction that $x$ and $y$ in $C$ satisfies $d(x,\alpha)=d(y,\alpha)$ and $d(x,\beta)=d(y,\beta)$. Then, since $d(x, \alpha) = d(x,u) + d(u,\alpha)$ and $d(y,\alpha) = d(y,u) + d(u,\alpha)$, we have $d(x,u)=d(y,u)$. Similarly $d(x,v)=d(y,v)$, a contradiction with Lemma~\ref{cycle1}.
\end{proof}

\begin{lemme}
\label{tree}
Let $G$ be an odd unicyclic graph.
If, for any $u \in V(C)$, $T_u$ is a rooted tree, then for all $e \in E(C)$, $\dim(G) \leq \dim(G-e)$.
\end{lemme}

\begin{proof}
Let $e \in E(C)$ and $S$ be a metric basis of $G-e$. In order to show that $S$ is still a resolving set of $G$, let us first prove that for every $u \in C$ since $T_u$ is a rooted tree $S \cap T_u \neq \emptyset$.\\ 

By definition of rooted tree, $T_u$ contains a vertex of degree $3$ in $G$. Let $r$ be such a vertex.
    By Lemma \ref{dimbranch}, at most one connected component of $G-e\setminus \{r\}$ does not contain  element of $S$. The tree $T_u$ contains at least two connected components of $G-e \setminus \{r\}$,  so $S \cap T_u \neq \emptyset$.

    Let $(x,y)$ be any pair of vertices. We prove that $S$ resolves $(x,y)$ in $G$.
    \begin{enumerate}
    \item Assume first that $x$ and $y$ are in the same component $T_u$ for some $u \in C$. 
    Let $\alpha \in S$ that resolves $(x,y)$ in $G-e$.
    If $\alpha \in T_u$ then $d_{G}(\alpha,x) = d_{G-e}(\alpha,x) \neq d_{G-e}(\alpha,y) = d_{G}(\alpha,y)$.
    
    We can assume that $\alpha \notin T_u$. Since $e \notin T_u$ and $u$ is a cut-vertex of $G$ and $G-e$, we have $d_{G-e}(u,w)=d_{G}(u,w)$ for every $w \in T_u$. For every $w \in T_u$, \[d_{G}(\alpha,w)=d_{G}(\alpha,u)+d_{G}(u,w)=d_{G}(\alpha,u)+d_{G-e}(u,w).\]
    Since $d_{G-e}(\alpha,w)=d_{G-e}(\alpha,u)+d_{G-e}(u,w)$ and $d_{G-e}(\alpha,x) \ne d_{G-e}(\alpha,y)$, we have $d_{G-e}(u,x) \ne d_{G-e}(u,y)$. Thus, $d_{G}(\alpha,x) \ne d_{G}(\alpha,y)$ and then $\alpha$ resolves $(x,y)$ in $G$.

    \item Suppose now $x$ and $y$ are in different components, respectively $ T_u$ and $ T_v$. Then, there exist $\alpha \in S \cap T_u$ and $\beta \in S \cap T_v$. Assume $d_G(\alpha,x)=d_G(\alpha,y)$ and  $d_G(\beta,x)=d_G(\beta,y)$. Then, $d_G(\alpha,x) \leq d_G(\alpha,u)+d_G(u,x)$ and $d_G(\alpha,y)=d(\alpha,u)+d_G(u,v)+d_G(v,y)$ so $ d_G(u,v)+d_G(v,y) \leq d_G(u,x)$. The symmetric relation is $ d_G(v,u)+d_G(u,x)\leq d_G(v,y)$. Summing the two gives $2d_G(u,v) \leq 0$ which is a contradiction since $u \neq v$.\qedhere
\end{enumerate} 
\end{proof}

Next, we prove by case distinction the following result.

\begin{theorem}\label{odd_unicyclic}
Any odd unicyclic graph $G$ satisfies $\dim(G) \leq Z(G)$.
\end{theorem}

\begin{proof}
We make a case analysis on the structure of $G$ with $C$ being the unique cycle of $G$. \\

\noindent
     \textbf{Case 1: For every $u \in V(C)$, $T_u$ is a rooted tree.} \\ 
    By Lemma~\ref{zero}, there is an edge $e$ such that  $Z(G-e) \leq Z(G)$. Since the $G-e$ is a tree, Corollary~\ref{treecomp} ensures that $\dim(G-e) \leq Z(G-e)$. By Lemma~\ref{tree}, $\dim(G) \leq \dim(G-e)$. The combination of these inequalities gives $\dim(G) \leq Z(G)$. \\
    
    \noindent
    \textbf{Case 2: For every $u \in V(C)$, $T_u$ is a rooted path.} \\ 
    We prove that any pair of vertices $\alpha$ and $\beta$ of $C$ at distance $k$ is a resolving set of $G$. \\
    Let ${\alpha,\beta}$ be such a pair of vertices. Let us first prove that for every $x \in G$, $d(x,\alpha)+d(x,\beta) \in \{k,k+1\}$ if and only if $x \in C$. Indeed as $d(\alpha,\beta)=k$ for any vertex $x$, $d(x,\alpha)+d(x,\beta) \geq k$. If $x \in C$, then either $x$ is on the path between $\alpha$ and $\beta$, and $d(x,\alpha)+d(x,\beta) =k$. Or $x$ is in the other part of the cycle and $d(x,\alpha)+d(x,\beta) = k+1$. If $x \notin C$, then let $y$ be the vertex of $C$ such that $x \in T_y$. Then, $d(x,\alpha)+d(x,\beta)=d(y,\alpha)+d(y,\beta)+2d(x,y) \geq 2 +d(y,\alpha)+d(y,\beta) \geq k+2$. This implies that a vertex on $C$ cannot have the same distance vector as a vertex in $V \setminus C$. \\
    By Lemma~\ref{cycle}, the set $\{\alpha,\beta\}$ resolves $C$. To conclude we  show that $\{ \alpha,\beta\}$ resolves $V \setminus C$.
    Let us prove that two vertices not in $C$ are resolved. Assume by contradiction that two vertices $x$ and $y$ are not resolved by $\{\alpha, \beta\}$, with $x \in T_u$ and $y \in T_v$ for $u$ and $v$ in $C$. If $u=v$, since $T_u$ is a rooted path, $\alpha$ resolves the pair $(x,y)$. From now on, we can assume that $u \neq v$. Assume by symmetry $d(x,u) \leq d(y,v)$, and let $z \in T_v$ be the vertex on $T_v$ such that $d(v,z)=d(x,u)$. Then, the pair $\{\alpha,\beta\}$ does not resolve the pair $(z,u)$. But $u \in C$ and the previous cases ensures that any pair of vertices with one vertex on $C$ is resolved by $\{\alpha,\beta\}$, a contradiction. So $\dim(G)=2$ and $G$ is not a path, so $Z(G) \geq 2$. \\ 
    
    \noindent
    \textbf{Case 3: There exists a unique $u \in V(C)$ such that $T_u$ is a rooted tree}.\\ 
    Let us prove that $\dim(G) \leq \dim(T_u)+1$ and $Z(G) \geq Z(T_u)+1$. Let $S$ be a metric basis for $T_u$ and $v \in C$ a vertex at distance $k$ from $u$.\\ 
    Let us show that $S \cup \{v\}$ is a resolving set of $G$. Let $\alpha \in S \cap T_u$, then $\{ \alpha,v\}$ resolves $G \setminus T_u$. Indeed otherwise, these two vertices would have the same distances to $u$ and $v$ which is impossible since in $G \setminus T_u$, for every $v \in C$, $T_v$ is a rooted path and by the claim in Case 2, $\{u,v\}$ is a resolving set for $G \setminus (T_u \setminus \{u\})$. 
    Two vertices on $T_u$ are resolved since $S$ is a metric basis for $T_u$ and $u$ is a cut-vertex.
    Let $x \in T_w$ for some $w \in C$, and $y \in T_u$. By triangular inequality, $d(\alpha,y) \leq d(y,u)+d(u,\alpha)$.
    If $d(\alpha,y)=d(\alpha,x)$, then $d(x,u) \leq d(y,u)$ as $d(\alpha,x)=d(\alpha,u)+d(u,x)$. If $d(x,v)=d(y,v)$, then as $d(y,v)=d(y,u)+d(u,v) \geq d(x,u)+d(u,v)$. We get $d(x,v) \geq d(x,u)+d(u,v)$. Removing $d(x,w)$ on both side gives $d(w,v) \geq d(w,u)+d(u,v)$ which is impossible since $d(u,v)=k$ and $d(w,v) \leq k$. So $x$ and $y$ have different codes and $\dim(G) \leq \dim(T_u)+1$. \\
    Let $Z$ be a minimal zero forcing set of $G$. If $Z$ contains $u$, then it should contain at least another vertex in $G \setminus T_u$. Since the restriction of $Z$ to $T_u$ is a forcing set for $T_u$, we have $Z(G) \geq Z(T_u)+1$. 
    So we can assume that $u \notin Z$. Consider a sequence of color change rule that turns $u$ into black. Either $u$ is forced by a vertex in $G \setminus T_u$. Since $(G \setminus T_u) \cup \{u\}$ contains a cycle, there are at least two vertices in $Z \cap (G \setminus T_u)$ and $Z \cap T_u  \cup \{u\}$ is a forcing set of $T_u$. So $Z(G) \geq Z(T_u)+1$. Otherwise $u$ is forced by a vertex of $T_u$. Then, there is at least one vertex in $Z \cap (G \setminus T_u)$ and $Z \cap T_u  $ is a forcing set of $T_u$ so $Z(G) \geq Z(T_u)+1$. So, $\dim(G) \leq \dim(T_u)+1 \leq Z(T_u)+1 \leq Z(G)$.\\
    
    For the remaining cases we use the following process: we exhibit an edge $e$ such that $Z(G-e) \leq Z(G)$ (by Lemma \ref{zero} or \ref{deg1}). Then, find a vertex $z$ in $G-e$ such that $z$ is an interior degree-2 vertex or a major vertex with terminal degree $0$ or $1$. By Lemma \ref{dimtree}, $\dim(G-e)< Z(G-e)$ and by Lemma \ref{lemme:vardim}, $\dim(G) \leq \dim(G-e)+1$, so $\dim(G) \leq Z(G)$. We just give the construction of $e$ and $z$. \\ 
    
    \noindent
    \textbf{Case 4: There exists $u,v,w,x \in C$ in this order (not necessarily adjacent) such that $T_u$ and $T_w$ are rooted trees and $T_v$ and $T_x$ are rooted paths.} \\
    Let $e \in E(C)$ such that $Z(G-e) \leq Z(G)$. Such an edge exists by Lemma \ref{zero}. In $G-e$, either $v$ or $x$ is on the path between $u$ and $w$. Let $z$ be this vertex, $z$ is an interior degree-2 vertex or a major vertex with terminal degree $0$ or $1$ in $G-e$.
    
    \item \textbf{Case 5: There exist $u$ and $v$ adjacent with $T_u$ and $T_v$ rooted trees and $w$ with $T_w$ a rooted path.} 
    \begin{itemize}
        \item \textbf{$l_u=0$.} 
        Let $e \in E(C)$ such that $Z(G-e) \leq Z(G)$. Such an edge exists by Lemma \ref{zero}. In $G-e$, $u$ is an interior degree-2 vertex or a major vertex with terminal degree $0$ or $1$ so $z=u$.
        \item \textbf{$l_u=1$.} 
        Let $e \in E(C)$ adjacent to $u$ such that $Z(G-e) \leq Z(G)$. Such an edge exists by Lemma \ref{deg1}. If $e=uv$ then $w$ is an interior degree-2 vertex or a major vertex with terminal degree $0$ or $1$ in $G-e$ so $z=w$. If $e$ in the other edge in $E(C)$ adjacent to $u$ then $u$ is an interior degree-2 vertex or a major vertex with terminal degree $0$ or $1$ in $G-e$ so $z=u$.
        \item \textbf{$l_u \geq 2$ and $l_v \geq 2$.} 
        Let $e=uv$, then $Z(G-e) \leq Z(G)$. Indeed a minimal zero forcing set of $G$ is also a zero forcing set of $G-e$. Let $Z$ be a zero forcing set of $G$. The set $Z$ contains at least one terminal vertex of $u$ and one terminal vertex of $v$. The terminal vertices can turn black $u$ and $v$ so there is a sequence of forces for $G$ such that the edge $e$ is not a forcing edge. So $Z(G-e) \leq Z(G)$ by Lemma $\ref{rev}$.
        Then $w$ is an interior degree-2 vertex or a major vertex with terminal degree $0$ or $1$ in $G-e$ so $z=w$. 
    \end{itemize}

\end{proof}


\printbibliography
\newpage

\appendix

\section{Proof of Lemma~\ref{lemme:vardim}}\label{appendix}
The proof of this lemma in the paper of Eroh et al.~\cite{eroh2017comparison} contains a flaw. We present here a correction of the proof based on the same general ideas.
\lemmevardim*

\begin{figure}[ht]
    \centering
    \includegraphics[scale=0.7]{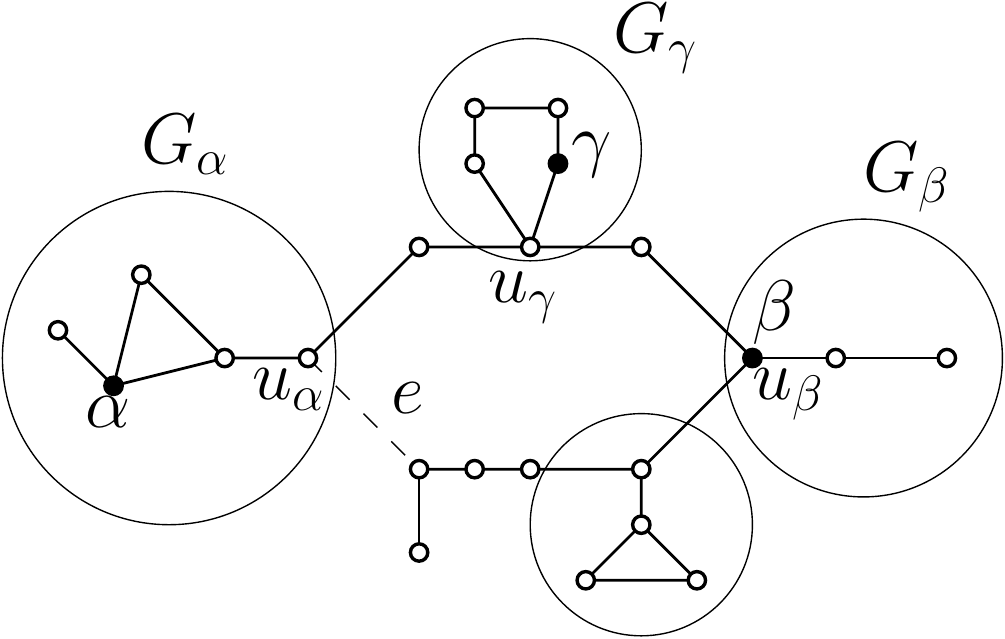}
    \caption{The set $\{\alpha,\beta,\gamma\}$ resolves any pair of vertices in two different circles}
    \label{fig4}
\end{figure}

\begin{proof}
Let $S$ be a metric basis of $G-e$ and $S_i=S \cap V_i$ for $1 \leq i \leq k$. Assume without loss of generality that $S_1 \neq \emptyset$. If $S_i=\emptyset$ for all $i\neq 1$, then $(G-e) \setminus G_1$ is a path and $S_1 \cup \{u_2\}$ is a resolving set of $G$. 
If there exists two non-empty subsets $S_i$ and $S_j$ such that $d_G(u_i,u_j)=\frac{k}{2}$, then let $\alpha_0 \in V(C) \setminus \{u_i,u_j\}$ and $S_0= \{\alpha_0,\alpha_i,\alpha_j\}$ where $\alpha_i \in S_i$ and $\alpha_j \in S_j$. Otherwise, let $i\neq 1$ be such that $S_i \neq \emptyset$, let $\alpha_0=u_{ \lfloor \frac{k}{2} \rfloor}$ and $S_0=\{ \alpha_0, \alpha_i,\alpha_1\}$ where $\alpha_1 \in S_1$ and $\alpha_i \in S_i$. Then, we prove that $S \cup \{ \alpha_0\}$ is a resolving set of $G$. \\
Let $v \in G_i$ and $w \in G_j$ with $i \neq j$ be two vertices in $G$. We show that one vertex in $S_0$ resolves the pair $(u,v)$. For simplicity, rename $S_0=\{ \alpha, \beta,\gamma\}$ with $\alpha \in G_{\alpha}$, $\beta \in G_{\beta}$ and $\gamma \in G_{\gamma}$ with $d(u_{\alpha},u_\beta)=\lfloor \frac{k}{2} \rfloor$ as in Figure \ref{fig4}.

Consider first the case where $u_v \notin \{u_\alpha,u_\beta,u_\gamma\}$ and $u_w \notin \{u_\alpha,u_\beta,u_\gamma\}$. Then, we have the following equalities:
\[d(v,\alpha)=d(w,\alpha) \text{ gives }d(v,u_v)+d(u_v,u_\alpha)+d(u_\alpha,\alpha)=d(w,u_w)+d(u_w,u_\alpha)+d(u_\alpha,\alpha);\]
\[d(v,\beta)=d(w,\beta) \text{ gives } d(v,u_v)+d(u_v,u_\beta)+d(u_\beta,\beta)=d(w,u_w)+d(u_w,u_\beta)+d(u_\beta,\beta);\]
\[d(v,\gamma)=d(w,\gamma) \text{ gives } d(v,u_v)+d(u_v,u_\gamma)+d(u_\gamma,\gamma)=d(w,u_w)+d(u_w,u_\gamma)+d(u_\gamma,\gamma).\]
Deleting the terms on the form $d(u_\alpha,\alpha)$ and equalizing we get 
\[d(v,u_v)-d(w,u_w)=d(u_w,u_\alpha)-d(u_v,u_\alpha)=d(u_w,u_\beta)-d(u_v,u_\beta)=d(u_w,u_\gamma)-d(u_v,u_\gamma).\]
If $d(v,u_v)-d(w,u_w)=0$, then, by Lemma~\ref{cycle3}, $u_v=u_w$ as $u_v$ and $u_w$ are at the same distance to three points on the cycle. Else, we have 
\[ d(u_\alpha,u_w)+d(u_w,u_\beta) - d(u_\alpha,u_v)+d(u_v,u_\beta) =2(d(v,u_v)-d(w,u_w)),\] which is a contradiction as $d(u_\alpha,u_\beta)=\lfloor \frac{k}{2} \rfloor$, the difference is in $\{-1,0,1\}$. So we get $u_v=u_w$. \\

Consider now the case of one vertex $u_v$ or $u_w$ is equal to $u_\alpha$ or $u_\beta$. Assume without loss of generality that $u_v=u_\alpha$ and $u_w \neq u_\alpha$. Then 
$d(v,\alpha) \leq d(v,u_v)+d(u_\alpha,\alpha)$ and as $d(w,\alpha)=d(w,u_w)+d(u_w,u_\alpha)+d(u_\alpha,\alpha)$, we get:
\[ d(w,u_w)+d(u_w,u_\alpha) \leq d(v,u_v). \]
Consider now the distances to $\beta$: $d(v,\beta)=d(v,u_v)+\lfloor \frac{k}{2} \rfloor+d(u_\beta,\beta)$ and $d(w,\beta) \leq d(w,u_w)+d(u_w,u_\beta)+d(u_\beta,\beta)$. As $d(v,\beta)=d(w,\beta)$:
\[d(v,u_v)+\lfloor \frac{k}{2} \rfloor \leq d(w,u_w)+d(u_w,u_\beta). \]
Thus, $d(v,u_v) \leq d(w,u_w)+d(u_w,u_\beta) -\lfloor \frac{k}{2} \rfloor $, and so 
\[  d(u_w,u_\alpha) \leq d(u_w,u_\beta) -\lfloor \frac{k}{2} \rfloor. \]
We assume $u_w \neq u_\alpha$ so $d(u_w,u_\alpha) \geq 1$ and $d(u_w,u_\beta) \leq \lfloor \frac{k}{2} \rfloor$ by definition of $k$. We get a contradiction as $d(u_w,u_\beta) \geq 0$. \\
The last case is $u_v = u_\gamma$ and $u_w \notin \{u_\alpha,u_\beta,u_\gamma\}$. We have 
$d(v,u_v)+d(u_v,u_\alpha)=d(w,u_w)+d(u_w,u_\alpha)$ and  $d(v,u_v)+d(u_v,u_\beta)=d(w,u_w)+d(u_w,u_\beta)$ and summing the two equalities gives:
\[ 2d (v,u_v)+d(u_v,u_\alpha)+d(u_v,u_\beta) = 2d (w,u_w)+d(u_w,u_\alpha)+d(u_w,u_\beta),\] which implies $d(v,u_v)=d(w,u_w)$.
Then, by triangular inequality $d(v,\gamma) \leq d(v,u_v)+d(u_\gamma,\gamma)$. Since $d(v,\gamma)=d(w,\gamma)$, we get $d(w,u_w)+d(u_w,u_\gamma) \leq d(v,u_v)$ and thus $d(u_w,u_\gamma)\leq 0$ which is a contradiction. \\
We proved that, if $v$ and $w$ are not in the same subgraph, then one vertex in $S_0$ resolves them. So, if $v$ and $w$ are in $G_i$ for some $i$, by definition of $S$, there exists $\mu \in S$ which resolve $(v,w)$ in $G-e$. Hence, $\mu$ still resolves $(v,w)$ in $G$: If $\mu \in G_i$, then the distances are the same in $G$ and in $G-e$. If $\mu \notin G_i$, then by hypothesis $d_{G-e}(v,\mu) \neq d_{G-e}(w,\mu)$. By decomposition 
$d_{G-e}(v,u_v) +d_{G-e}(u_v,\mu) \neq d_{G-e}(w,u_w) +d_{G-e}(u_w,\mu)$ so $d_{G-e}(v,u_v) \neq d_{G-e}(w,u_w)$. As $d_{G-e}(v,u_v)=d_G(v,u_v)$
and $d_{G-e}(w,u_w)=d_G(w,u_w)$, we get $d_{G}(v,u_v) +d_{G}(u_v,\mu) \neq d_{G}(w,u_w) +d_{G}(u_w,\mu)$. So $\mu$ resolves $(v,w)$.
Finally, $S \cup \{\alpha_0\}$ is a resolving set of $G$, and so $\dim(G) \leq \dim(G-e)+1$.

\end{proof}

\end{document}